\documentclass[11pt,reqno,a4paper]{amsart}  
\usepackage{amssymb,amsthm,amsmath}
\usepackage[mathscr]{eucal} 
\usepackage[latin1]{inputenc}

\usepackage{lmodern}
\usepackage[T1]{fontenc}
\usepackage[english]{babel}

\usepackage[linktoc=page]{hyperref}

\usepackage[top=3.2cm, bottom=3.6cm, left=2.7cm, right=2.7cm]{geometry}

\renewcommand{\leq}{\leqslant}
\renewcommand{\geq}{\geqslant}

\newcommand{\No}{\mathrm{N}} 
\newcommand{\Oc}{\mathcal{O}} 
\newcommand{\Repr}{\mathcal{R}}

\newcommand{\Qppq}{\Qq(\sqrt{q})}

\newcommand{\Kcl}{\mathcal{K}}  
\newcommand{\Hcl}{\mathcal{H}}  
\newcommand{\Hhat}{\widehat{\Hcl}}  
\newcommand{\Hhatc}{\widehat{\Hcl}_1}

\newcommand{\wpsi}{\widetilde{\psi}}

\newcommand{\chiq}{\chi^{(q)}}  

\newcommand{\wc}{c_0}

\newcommand{\eps}{\varepsilon}

\newcommand{\aF}{\mathfrak a}
\newcommand{\bF}{\mathfrak b}
\newcommand{\pF}{\mathfrak p}

\newcommand{\DF}{\mathfrak F}

\def\stacksum#1#2{{\stackrel{{\scriptstyle #1}}
{{\scriptstyle #2}}}}

\newcommand{\Cc}{\mathbb{C}}

\newcommand{\Zz}{\mathbb{Z}}
\newcommand{\Rr}{\mathbb{R}}
\newcommand{\Qq}{\mathbb{Q}}

\newcommand{\mods}[1]{\,(\mathrm{mod}\,{#1})}

\DeclareMathOperator{\Reel}{Re}
\DeclareMathOperator{\li}{li}

\DeclareMathOperator{\GL}{GL}
\DeclareMathOperator{\SL}{SL}
\DeclareMathOperator{\ord}{ord}

\newcommand{\demi}{{\textstyle{\frac{1}{2}}}}

\numberwithin{equation}{section}

\theoremstyle{plain}
\newtheorem{theorem}{Theorem}[section]

\newtheorem{lemma}[theorem]{Lemma}
\newtheorem{corollary}[theorem]{Corollary}

\newtheorem{proposition}[theorem]{Proposition}

\theoremstyle{remark}

\theoremstyle{definition}
\newtheorem{remarkf}[theorem]{Remark}
\newtheorem*{remf}{Remark}

\newtheorem{definition}[theorem]{Definition}
\newtheorem*{ack}{Acknowledgements}

\begin{document}


\title[Primes represented by binary quadratic forms: The average distribution]{On the average distribution of primes represented by binary quadratic forms}

\date{\today}

\author[Jakob Ditchen]{Jakob J. Ditchen}
\address{ETH Z\"urich -- D-MATH\\
  R\"amistrasse 101\\
  8092 Z\"urich\\
  Switzerland} 
 \email{jakob.ditchen@math.ethz.ch}


\begin{abstract}
 We investigate the average distribution of primes represented by positive definite integral binary quadratic forms, the average being taken over negative fundamental discriminants in long ranges. In particular, we prove corresponding results of Bombieri--Vinogradov type and of Barban--Davenport--Halberstam type, although with shorter ranges than in the original theorems for primes in arithmetic progressions: The results imply that, for all $\eps>0$, the least prime that can be represented by any given positive definite binary quadratic form of discriminant $q$ is smaller than $|q|^{7+\eps}$ for all forms to ``most'' discriminants; moreover, it is even smaller than $|q|^{3+\eps}$ for ``most'' forms to ``most'' discriminants.
\end{abstract}

\maketitle
\setcounter{tocdepth}{1}
\tableofcontents


\section{Introduction and statement of the main results}\label{sec:intro}

 Results on the average distribution of prime numbers in arithmetic progressions have often proved to be suitable substitutes for conditional statements that rely on the Generalized Riemann Hypothesis -- and sometimes even surpass its direct consequences. The Bombieri--Vinogradov theorem~\cite{Bom65},\cite{Vin65,Vin66} and the Barban--Davenport--Halberstam theorem \cite{Bar66},\cite{DH66,DH66_corr}  from the 1960s are two of the most prominent and influential of these results. The first theorem states that, for each $A>0$, there exists a number  $B=B(A)$ such that
\begin{equation}\label{eq:BV_orig}
\sum_{q\leq Q} \max_\stacksum{a \leq q}{(a,q)=1} \Big| \pi(X;q,a) - \frac{\li(X)}{\varphi(q)}\Big| \ll_A X(\log X)^{-A}
\end{equation}
for $Q \leq X^{1/2}(\log X)^{-B}$; here $\pi(X;q,a)$ denotes the number of primes $p \leq X$ with\linebreak $p \equiv a \mods q$ for any pair $(a,q)$ of coprime integers, $\li(X)=\int_2^X \frac{1}{\log t}\, dt$ is the logarithmic integral and $\varphi(q)$ is the Euler totient function.  That is, the error term in the prime number theorem for arithmetic progressions is small -- as small as predicted by the Riemann Hypothesis~-- for all reduced residue classes,  ``on~average'' over moduli in about the same range of moduli in which the Riemann Hypothesis yields non-trivial results. The second theorem shows that the mean square of the error term is small for an even longer range of moduli if one averages over both moduli and their reduced residue classes: For each $A>0$, there exists a number  $B=B(A)$ such that
\begin{equation}\label{eq:BDH_orig}
\sum_{q\leq Q} \sum_\stacksum{a \leq q}{(a,q)=1} \left( \pi(X;q,a) - \frac{\li(X)}{\varphi(q)}\right)^2 \ll_A X^2(\log X)^{-A}
\end{equation}
for $Q \leq X(\log X)^{-B}$. 

\smallskip
Apart from arithmetic progressions, integral binary quadratic forms constitute the simplest\linebreak family of polynomials -- and, in fact, one of only very few families of polynomials in two variables -- that are known to represent infinitely many prime numbers unless there is an obvious obstacle by means of a common prime divisor of the coefficients. 
Analytic questions on primes which are representable by any \emph{fixed} binary quadratic form have been studied almost as extensively as analytic questions on primes in fixed arithmetic progressions: De la Vall\'ee Poussin's seminal work \cite{Pou96} on the prime number theorem does not only contain proofs for the prime number theorem in its ordinary form and for primes in arithmetic progressions but also for primes represented by positive definite binary quadratic forms. Moreover, the best known upper bounds for the error terms in both prime number theorems are essentially the same. 

\smallskip
In this paper we prove average distribution results of the shapes \eqref{eq:BV_orig} and \eqref{eq:BDH_orig} for primes that are representable by integral binary quadratic forms of \emph{various} negative fundamental discriminants in long ranges. 
In particular, we show: 
\begin{theorem}\label{thm:BV_01_PNT}
For any $Q \geq 1$, let $\DF(Q)$ denote the set of all negative fundamental discriminants $q \not \equiv 0 \mods 8$ with $|q| \leq Q$. 
 For any form class $C$ in the form class group $\Kcl(q)$ of any discriminant $q \in \DF(Q)$, let $e(C)=2$ if $C$ is of order at most two in $\Kcl(q)$ and $e(C)=1$ otherwise; moreover, let $\pi(X;q,C)$ be the number of primes $p\leq X$ that can be represented by all 
 binary quadratic forms in~$C \in \Kcl(q)$. Let $h(q)=|\Kcl(q)|$ denote the class number for~$q \in \DF(Q)$. 

For each 
 $A>0$ and each 
$\eps>0$, there exists a real number $B=B(A)$ such that 
  \begin{equation}\label{eq:my_bv_PNT}
  \sum_{q \in \DF(Q)} \max_{C \in \Kcl(q)} \Big|\pi(X;q,C)- \frac{\li(X)}{e(C) h(q)}\Big|\ll_{A,\eps} Q^{1/2}X(\log X)^{-A}
  \end{equation}
  if $ Q^{20/3+\eps} \leq X(\log X)^{-B}$. 
\end{theorem}
By giving up 
 control over the form classes, analogously to the Barban--Davenport--\linebreak Halberstam theorem, we may extend the range of the discriminants: 
\begin{theorem}\label{thm:bdh1}
   Let $A>0$ and $\eps>0$. Then 
	\begin{equation}\label{eq:bdh1}
	\sum_{q \in \DF(Q)} \sum_{C \in \Kcl(q)} \left(\pi(X;q,C)- \frac{\li (X)}{e(C) h(q)}\right)^2\ll_{A,\eps} Q^{1/2}X^2(\log X)^{-A}
	\end{equation}
	if $Q^{3 +\eps} \leq X(\log X)^{-2A-4}$.
\end{theorem}
\noindent The representability of the primes is therefore well distributed over all (Theorem \ref{thm:BV_01_PNT}) or almost all (Theorem~\ref{thm:bdh1}) form classes to almost all negative fundamental discriminants $q \not\equiv 0 \mods 8$ in long ranges.

\begin{remarkf}\label{rem:BV_comparison_to_trivial}
How do these results compare to ``trivial'' estimates? There is no estimate for primes represented by a given binary quadratic form which is as trivial as the estimate\linebreak 
$\pi(X;q,a)\leq \frac{X}{q}+1$ for primes in arithmetic progressions, where the right-hand side of the inequality is simply the number of positive integers up to $X$ in the given arithmetic progression (or this number plus one).  
However, the number of 
 integers $n \leq X$ that can be represented by any 
 binary quadratic form of discriminant $q \in \DF(Q)$ is~$\ll \frac{X}{\sqrt{|q|}}$ and this 
 can be proved by an elementary lattice point counting argument, so this may therefore be considered as  
 a suitable substitute for a completely trivial bound. Moreover, it is known that the class number $h(q)$ has the lower bound 
\begin{equation}\label{eq:lowerclassnumberbound}
|q|^{1/2} (\log |q|)^{-1} \ll h(q)
\end{equation}
 if the primitive real Dirichlet character modulo $|q|$ is not exceptional (i.e., if the associated \mbox{$L$-function} does not have a Landau--Siegel zero), and $|q|^{1/2-\eps} \ll_\eps h(q)$  for all $\eps>0$ if  it is exceptional. Since exceptional discriminants are very rare (see Proposition \ref{prop:landau-siegel1}), it is reasonable to use \eqref{eq:lowerclassnumberbound} and to consider 
\begin{equation*}
O\big( Q^{1/2} (\log Q)X \big)
\end{equation*}
as a ``trivial'' upper bound for the sum on the left-hand side of \eqref{eq:my_bv_PNT}. We improve on this by an arbitrary power of $(\log X)$ in Theorem \ref{thm:BV_01_PNT},  just as in the original Bombieri--Vinogradov theorem. Similarly, Theorem~\ref{thm:bdh1} saves an arbitrary power of $(\log X)$ over the corresponding easy estimate for the left-hand side of \eqref{eq:bdh1}.
\end{remarkf}

\begin{remarkf} 
Analogues of the Bombieri--Vinogradov and Barban--Davenport--Halberstam theorems have been investigated in various contexts in the past, 
but we are not aware of any prior results of the type that we consider in Theorem \ref{thm:BV_01_PNT} and Theorem~\ref{thm:bdh1}. 
With regard to the well-known connection between classes of integral binary quadratic forms and ideal classes of quadratic fields that we will exploit, 
 it is important to note that many results of Bombieri--Vinogradov type have already been proved for number fields, e.g.\ by Wilson (1969), Huxley (1971), Fogels (1972), Johnson (1979) and Hinz (1988) (see the references given in \cite[\S 7.4.12]{Nar04}), but all these results have examined cases in which the number field is fixed; this is not useful in our case. The only results that have hitherto been proved for varying number fields are \cite{MM87} and the recent generalization \cite{MP13}; the fields in these works are of the form $K(\zeta_q)$ where $K$ is a fixed number field, $\zeta_q$ is a primitive $q$-th root of unity and $q$ varies. This case, which uses the large sieve inequality for Dirichlet characters, is also quite different to the situation of Theorem \ref{thm:BV_01_PNT}, which requires a large sieve inequality that takes ideal class group characters of various fields  simultaneously into account (see Lemma \ref{lem:lsi_complex}) and comes up with other subtle differences. 
\end{remarkf}

\begin{remarkf}
It seems that the condition $q \not\equiv 0 \mods 8$ can be dropped quite easily, but the proof then requires slightly more care when dealing with primitive real Dirichlet characters (since there are two such characters modulo $8$) and the application of the functional equation for Rankin--Selberg convolutions of holomorphic cusp forms in Section \ref{sec:largesieve} becomes more technical. Also, both theorems should be provable along similar lines for even more general negative discriminants, but many parts of the proof appear to become much more laborious if not intractable due the peculiarities that are linked to the square factors of non-fundamental discriminants, especially in the above-mentioned functional equation (consider, however, also Remark~\ref{rem:li-langlands}).
\end{remarkf}

\begin{remarkf}\label{rem:conditional_ranges}
Assuming the Lindel\"of Hypothesis (for Rankin--Selberg convolutions of holomorphic cusp forms of weight one), the exponent $\frac{20}{3}+\eps$ in Theorem \ref{thm:BV_01_PNT} may be replaced by $2+\eps$ and the exponent $3+\eps$ in Theorem~\ref{thm:bdh1} may be replaced by $1+\eps$.
\end{remarkf}

\begin{remarkf}
One reason for the comparatively short ranges that are, for now, admissible for the discriminants in our results (compared to the ranges of the moduli in the original theorems for arithmetic progressions) may be found  in the fact that the size of a form class group is much smaller than the corresponding discriminant. 
This offers therefore less potential for possible cancellation effects than in the case of arithmetic progressions where the number of reduced residue classes of a modulus is usually 
 only slightly smaller than  the modulus itself.
\end{remarkf}

In order to prove Theorem \ref{thm:BV_01_PNT} we will largely follow Gallagher's proof of the original Bombieri--Vinogradov theorem as presented by Bombieri in \cite[\S 7]{Bom87}. The key ingredients will be:
\begin{enumerate}
\item[(1)]
Dedekind's bijection between form classes and ideal classes in imaginary quadratic fields.
	\item[(2)]
	A new large sieve inequality for complex class group characters, which we prove via Rankin--Selberg convolutions of holomorphic cusp forms of weight one; see Sections \ref{sec:largesieve} and \ref{sec:complexchar1}.
	\item[(3)]
	The original Bombieri--Vinogradov theorem, which we use to estimate the contribution coming from real class group characters; see Section \ref{sec:realchar1}.
	\item[(4)]
	Landau's theorem on the scarcity of exceptional moduli, that is, the rarity of integers $q$ for which there could possibly exist a Dirichlet character modulo $q$ whose associated $L$-function has a Landau--Siegel zero; see Proposition \ref{prop:landau-siegel1}.
  \item[(5)]
	A result of Siegel--Walfisz type for ideal class group characters; see Proposition~\ref{prop:sw-goldstein}.
\end{enumerate}
The proof of Theorem~\ref{thm:bdh1} is similar, but the fifth ingredient above will be replaced by a direct appeal to Blomer's Siegel--Walfisz theorem for binary quadratic forms \cite{Blo04}.  
In fact, we prove in Theorem \ref{thm:bdh_general}  that general arithmetic functions exhibit an ``average behaviour'' with respect to the representability of integers by form classes -- for most form classes to most discriminants in long ranges -- if the functions satisfy Siegel--Walfisz conditions for both arithmetic progressions and form classes (and an additional technical condition).

\medskip
An easy application of these theorems yields upper bounds that ``usually'' hold for the size of the least prime represented by any  given positive definite binary quadratic form:
\begin{corollary}\label{cor:leastprime}
Let $\DF$ be the set of all negative fundamental discriminants $q \not \equiv 0 \mods 8$. For each $q \in \DF$
  and each form class ${C\in \Kcl(q)}$, let $p(q;C)$ denote  the least prime which is representable by all binary quadratic forms in~$C$.
	
1. For each $\eps>0$, 
the upper bound 
\begin{equation}\label{eq:linnik1}
\max_{C \in \Kcl(q)} p(q;C) \leq |q|^{\frac{20}{3}+\eps}
\end{equation}
may only fail for fundamental discriminants $q$ lying in a set $V=V(\eps) \subset \DF$ that has asymptotic density $0$ in $\DF$.

2. Moreover, for each $\eps > 0$, there exists a subset $S=S(\eps)$ of $\DF$
 such that $S$ has asymptotic density~$1$ in $\DF$, and
\begin{equation}\label{eq:linnik2}
\lim_{n \to \infty}\frac{|\{C \in \Kcl(q_n) \mid p(q_n;C) \leq |q_n|^{3+\eps}\}|}{h(q_n)}=1
\end{equation}
holds for each sequence $(q_n)$ in $S$ with $|q_n| \to \infty$ as $n \to \infty$.
\end{corollary}
These bounds give the first explicit exponents (although only ``on average'') for the bound $p(q;C)\ll |q|^L$ that is known to hold with some absolute constant $L$ for all negative fundamental discriminants $q$ and all form classes $C \in \Kcl(q)$. We will discuss in Section \ref{sec:leastprime} how the bound \eqref{eq:linnik1} could be potentially improved for the special forms of the shape $x^2+ny^2$ for at least almost all positive squarefree integers $n$. 

\medskip
\begin{ack}
This work 
is based on part of  the author's doctoral dissertation. 
He would like to thank his supervisor, Professor Emmanuel Kowalski, for suggesting the investigation of this problem and for much valuable advice.
\end{ack}


\section{Definitions and preliminaries on form class groups and ideal class groups}\label{sec:prelim_0}
We introduce in this section some basic definitions (some of which have already appeared in Section \ref{sec:intro} and will be repeated for convenience) and review certain properties concerning discriminants, form class groups and ideal class groups, which will be used in the subsequent sections. 

\smallskip
We will denote the set of all negative fundamental discriminants $q \not\equiv 0 \mods 8$ by $\DF$, i.e.\  
\begin{equation*}
\begin{aligned}
\DF=&\big\{d \in \Zz \mid d<0,\, d \equiv 1 \mods 4 \text{ and $d$ is squarefree}\big\} 
\\& 
\cup \big\{d \in \Zz \mid d<0,\, d \equiv 0 \mods 4 \text{ and ${\textstyle{\frac{d}{4}}} \equiv 3 \mods 4$ is squarefree}\big\}
\end{aligned}
\end{equation*}
and, for all $Q \geq 1$, we write $\DF(Q)$ for the set of all $q \in \DF$ with $|q|\leq Q$. 

\smallskip
Two binary quadratic forms $f$ and $g$ (which will always be assumed to be integral, primitive and positive definite in this paper) of discriminant $q \in \DF$ are called equivalent if there exists $\gamma \in \SL(2,\Zz)$ such that $f(x,y)=g(\gamma(x,y))$ for all $x,y \in \Zz$. Dirichlet defined a composition on the set of the resulting equivalence classes, which are called form classes; it turns the set into an abelian group, the form class group $\Kcl(q)$, whose cardinality, the class number $h(q)$, is known to be always finite. Equivalent forms represent the same numbers and we may therefore define the set   
\[
\Repr(q,C)=\big\{n \in \Zz \mid \forall f \in C \, \exists x,y \in \Zz:\, f(x,y)=n  \big\}
\] 
 for all $q \in \DF$ and all form classes $C\in \Kcl(q)$. See \S2 and \S3 in \cite{Cox97} for proofs and details. 

\smallskip
\noindent For each $q \in \DF$, we define: 
\begin{itemize}
\item
$\Oc(q)$, the ring of integers of $\Qq(\sqrt{q})$;
\item
$Z(q)$,  the set of 
non-zero integral $\Oc(q)$-ideals;
\item
$\No(\aF)$, the norm of the ideal $\aF \in Z(q)$, i.e.\ the size of the quotient ring $\Oc(q) /\aF$ (the dependence on $q$ is suppressed);
\item
$\Hcl(q)$,  the quotient of the group of invertible fractional $\Oc(q)$-ideals by the subgroup of principal fractional  $\Oc(q)$-ideals, i.e.\ the {ideal class group} 
of $\Oc(q)$;
\item
$\Hhat(q)$, the group of ideal class group characters $\chi:\ \Hcl(q) \to  \{z \in \Cc:\, |z|=1\}$; we write $\chiq_0$ for the trivial character and, overloading the notation, we define   
$\chi(\aF):=\chi(C)$ 
for all $\chi \in \Hhat(q)$ and all ideals $\aF \in Z(q)$, where $C \in \Hcl(q)$ is the ideal class of $\aF$.
\end{itemize}

\smallskip
\noindent Binary quadratic forms and ideal classes are linked through the following classical result: 
\begin{theorem}[Dedekind]\label{thm:cox77}
For every $q \in \DF$, there exists an isomorphism 
\[
B_q:\ \Kcl(q) \to \Hcl(q).
\]
In particular, we have
\[
h(q)=|\Kcl(q)|=|\Hcl(q)|.
\] 
Moreover, a positive integer $m$ is represented by the binary quadratic forms in the class $C \in~\Kcl(q)$  if and only if there exists an ideal $\aF \in B_q(C)$ such that $\No(\aF)=m$. 
\end{theorem}
\noindent A proof can be found in \cite[Theorem 7.7]{Cox97}, for example. 

\smallskip
For all (positive or negative) fundamental discriminants $q \not\equiv 0 \mods 8$, let $\chi_q$ denote 
 the unique primitive real Dirichlet character modulo $|q|$  
(there are two primitive real Dirichlet characters if $q \equiv 0 \mods 8)$; it is given by the Kronecker symbol~$(\frac{q}{\cdot})$ (see \cite[\S 3.5]{IK04}).  
 For each rational prime $p$, the number of solutions $m \mods p$ to $m^2 \equiv q \mods p$ equals $1+\chi_q(p)$ and one can easily show (see \cite[Proposition 5.16]{Cox97}, for example): 
\begin{itemize}
\item
If $\chi_q(p)=0$, i.e.\ if $p$ divides $q$, then $p$ ramifies in $\Oc(q)$, i.e.\ $p\Oc(q)=\pF^2$ for some prime ideal $\pF$ of $\Oc(q)$ and $\No(\pF)=p$;
\item
if $\chi_q(p)=1$, then $p$ splits in $\Oc(q)$, i.e.\ $p\Oc(q)=\pF_1\pF_2$ for two distinct prime ideals $\pF_1,\pF_2$ of $\Oc(q)$ and $\No(\pF_1)=\No(\pF_2)=p$;
\item
if $\chi_q(p)=-1$, then $p$ remains prime in $\Oc(q)$, i.e.\ $p\Oc(q)=\pF$ is a prime ideal in $\Oc(q)$ and $\No(\pF)=p^2$.
\end{itemize}
It follows with Theorem \ref{thm:cox77} that, if $n=p^\ell$ for a prime $p$ and a positive integer $\ell$ and if $n$ can be represented by the forms in the class $C \in \Kcl(q)$, then 
\begin{equation}\label{eq:weight_w-Cn_for_primepowers}
w(C,n):=\sum_\stacksum{\aF \in {B}_q(C)}{\No(\aF) =n} 1 = 
\begin{cases}
\ell +1 \quad &\text{if $\chi_q(p)=1$ and $\ord(C) \leq 2$ in $\Kcl(q)$},\\
1 \quad &\text{if $\chi_q(p)=1$ and $\ord(C) > 2$ in $\Kcl(q)$},\\
1 \quad &\text{if $\chi_q(p)=0$},\\
1 \quad &\text{if $\chi_q(p)=-1$ and $\ell$ is even}.
\end{cases}
\end{equation}
Only a small set of primes ramifies in $\Oc(q)$. Thus, if the number $w(C,p)$ is positive, it will usually be given by one of the first two cases in \eqref{eq:weight_w-Cn_for_primepowers}. For further use, we therefore put 
\begin{equation}\label{eq:definition_of_eK}
e(C)=
\begin{cases}
2 \quad &\text{if $\ord(C) \leq 2$ in $\Kcl(q)$},\\
1 \quad &\text{if $\ord(C) > 2$ in $\Kcl(q)$}.
\end{cases}
\end{equation}
For all $q \in \DF$ and all $X \geq 1$, we thus have
\begin{equation}\label{eq:weight_w-Cn_on_average}
\sum_{C \in \Kcl(q)}\sum_\stacksum{p\leq X}{p \in \Repr(q,C)} w(C,p)=
\sum_{C \in \Kcl(q)}e(C)\sum_\stacksum{p\leq X}{p \in \Repr(q,C)} 1 - \sum_\stacksum{p\leq X}{p \mid q} 1=
 \sum_{p \leq X} (1+ \chi_q(p)).
\end{equation}


\section{A large sieve inequality for complex ideal class group characters}\label{sec:largesieve}

The theorems of Bombieri--Vinogradov and Barban--Davenport--Halberstam are built on the following large sieve inequality for Dirichlet characters (see \cite[Theorem 7.13]{IK04}, for example):
\begin{lemma}[Large sieve inequality for Dirichlet characters]\label{lem:lsi_org}
For any positive integers $Q$ and $N$ and any complex numbers $(a_n)_{n \leq N}$, we have
\[
\sum_{q \leq Q}\  \sideset{}{^*}\sum_{\chi \mods q} \Big|\sum_{n \leq N} a_n \chi(n)\Big|^2 \leq (Q^2+N) \sum_{n \leq N} |a_n|^2,
\]
where $\sideset{}{^*}\sum $ means that the sum is taken over primitive Dirichlet characters only.
\end{lemma}

Due to the close relationship between the form class group $\Kcl(q)$ and the ideal class group $\Hcl(q)$ (of the imaginary quadratic field $\Qq(\sqrt{q})$) for each discriminant $q$ (see 
Section~\ref{sec:prelim_0}), 
the ideal class group characters 
$\chi \in \Hhat(q)$ 
play a similar role in the study of primes represented by binary quadratic forms as Dirichlet characters do in the study of primes in arithmetic progressions. 

Real class group characters arise from Dirichlet convolutions of real Dirichlet characters (compare Section~\ref{sec:realchar1}) and can be handled by means of Lemma \ref{lem:lsi_org}. 
 Since this is not the case for complex class group characters, the following large sieve inequality for such characters will be essential in the proofs of Theorems \ref{thm:BV_01_PNT} and \ref{thm:bdh1}: 
\begin{lemma}[Large sieve inequality for complex ideal class group characters]\label{lem:lsi_complex}
For each $Q \geq 1$, let $\DF(Q)$ be the set of all negative fundamental discriminants $q \not \equiv 0 \mods 8$  with $|q| \leq Q$. Set $\Hhatc(q)=\{\chi \in \Hhat(q) \mid \chi^2 \neq \chiq_0\}$ for all $q \in \DF(Q)$
and 
\[
\Hhatc(Q)=\bigcup_{q \in \DF(Q)} \Hhatc(q).
\] 
For each $q \in \DF(Q)$, each $\chi \in \Hhat_1(q)$ and each positive integer $n$, we set 
\begin{equation}\label{eq:def_lambda}
\lambda_\chi(n)=\sum_\stacksum{\aF \in Z(q)}{\No(\aF)=n} \chi(\aF).
\end{equation} 
Then
\begin{equation}\label{eq:lsi_complex}
\sum_{\chi \in \Hhatc(Q)} \Big| \sum_{n \leq N} a_n \lambda_\chi(n)\Big|^2 \ll_\eps 
\left(N (\log N)^3+N^{1/2}(\log N)Q^{5/2+\eps}\right) \sum_{n \leq N} |a_n|^2
\end{equation}
for all complex numbers $(a_n)_{n \leq N}$ and all $\eps>0$, $Q \geq 1$ and $N \geq 3$. 
\end{lemma}
The proof will in essence 
follow the proof of a similar mean-value estimate for automorphic representations by Duke and Kowalski \cite[Theorem 4]{DK00}. Apart from standard techniques that are often used in proofs of large sieve inequalities (like the duality principle), Rankin--Selberg theory is a key ingredient here. In contrast to the result in \cite{DK00}, which depends on (deep) facts from the theory of automorphic representations, we may use ``classical'' results about holomorphic cusp forms by appealing to Li's functional equation for $L$-functions that are associated to Rankin--Selberg convolutions of holomorphic cusp forms \cite{Li79}. This functional equation is quite complicated to use in its general form, but rather simple in our case of fundamental discriminants $q \not\equiv 0 \mods 8$.

\begin{remf}
There exist other large sieve inequalities for algebraic number fields. For instance, Schumer's \cite{Sch86} general  inequality with  explicit dependence of the constants on the parameters of the underlying \emph{fixed} field yields 
\begin{equation*}
\sum_{\chi \in \Hhatc(q)} \bigg| \sum_\stacksum{\aF \in Z(q)}{\No(\aF) \leq N} c(\aF) \lambda_\chi(n)\bigg|^2 \ll 
(\log |q|)\big(|q|+ |qN|^{1/2} +N\big) \sum_\stacksum{\aF \in Z(q)}{\No(\aF) \leq N} |c(\aF)|^2
\end{equation*}
for any fixed $q \in \DF$ and any function $c$ on $Z(q)$. However, the mean-value results of the next sections consider situations where the underlying number fields vary and therefore also require a large sieve inequality which has an extra averaging over the discriminant. To our knowledge, Lemma \ref{lem:lsi_complex} is the first large sieve inequality for \emph{varying} number fields.
\end{remf}

\begin{proof}[Proof of Lemma \ref{lem:lsi_complex}]
Let $\phi$ be a smooth majorant of the characteristic function of the\linebreak interval~$[0,N]$, i.e.\ a positive $C^\infty$ function on $[0,+\infty)$ with compact support, $0 \leq \phi \leq 1$ and $\phi(n)=1$ for $n \leq N$. 
For all $\chi_1 \in \Hhatc(q_1)$, $\chi_2 \in \Hhatc(q_2)$ with $q_1, q_2 \in \DF(Q)$, let $\chi_{1,2}$ be the product of the (unique) primitive real Dirichlet characters modulo $|q_1|$ and $|q_2|$; $\chi_{1,2}$ is therefore a real Dirichlet character modulo the least common multiple of $q_1$ and $q_2$. Set
\begin{align*}
S_N(\chi_1,\chi_2)&=\sum_{n \geq 1} \lambda_{\chi_1}(n) \overline{\lambda_{\chi_2}(n)} \phi(n/N),\\
L(s;\chi_1,\chi_2)&=\sum_{n \geq 1} \lambda_{\chi_1}(n) \overline{\lambda_{\chi_2}(n)} n^{-s},\\
L_{RS}(s;\chi_1,\chi_2)&=L(2s,\chi_{1,2}) L(s;\chi_1,\chi_2).
\end{align*}
The first $L$-function is the ``na\"ive'' convolution $L$-series of $\lambda_{\chi_1}(n)$ and $\overline{\lambda_{\chi_2}(n)}$, the second one is known as the \emph{Rankin--Selberg convolution} $L$-function.  
By the Mellin inversion theorem, we have
\begin{equation}\label{eq:mellin_SN}
S_N(\chi_1,\chi_2)=\frac{1}{2\pi i} \int_{(2)} N^s \widehat\phi(s) L(s;\chi_1,\chi_2) \,ds=\frac{1}{2\pi i} \int_{(2)} N^s \widehat\phi(s) \frac{L_{RS}(s;\chi_1,\chi_2)}{L(2s,\chi_{1,2})} \,ds,
\end{equation}
where 
\[
\widehat{\phi}(s)=\int_{0}^{+\infty} \phi(x) x^{s-1}\, dx
\]
denotes the Mellin transform of $\phi$ (see \cite[\S2.3]{Kow04} for this and the following basic properties of smooth cutoff functions and Mellin transforms). We would like to shift the line of integration on the right-hand side of \eqref{eq:mellin_SN} as far to the left as possible. Herefore, we need to know the growth behaviour of the functions in this integral: By the choice of $\phi$, its Mellin transform $\widehat{\phi}$ decays faster than any polynomial in all vertical strips of the complex plane. Furthermore, we have
\begin{equation}\label{eq:llll}
\frac{1}{L(2(\sigma+it),\chi_{1,2})}  \ll \zeta(2\sigma+2it) \ll \frac{1}{2\sigma-1}
\end{equation}
uniformly in $t \in \Rr$ if $\sigma >\demi$. 
As for the Rankin-Selberg $L$-function $L_{RS}(s;\chi_1,\chi_2)$, we consider the functions 
\begin{equation*}
f_j(z)=\sum_{n \geq 1} \lambda_{\chi_j}(n)e^{2 \pi i z n}\qquad (j=1,2)
\end{equation*}
 on the complex upper half plane. Since the involved class group characters $\chi_j$ are not real, we know (see \cite[\S14.3]{IK04}, for example) that the functions $f_j$ are normalized primitive holomorphic cusp forms of weight one, level $q_j$ and nebentypus $\chi_{q_j}$, the primitive real Dirichlet character modulo $|q_j|$. Therefore we also know from classical Rankin--Selberg theory (see \cite[Theorem 3.1]{Li79}) that $L_{RS}(s;\chi_1,\chi_2)$ is an entire function if $f_1 \neq f_2$ or, equivalently, if $\chi_1 \neq \chi_2$. 
In this case, it is therefore possible to shift the line of integration to $\Reel(s)=\demi+\alpha$ with $\alpha=(\log N)^{-1}$. Thus,
\[
S_N(\chi_1,\chi_2) \ll  \int_{(1/2+\alpha)} N^s \widehat\phi(s) \frac{L_{RS}(s;\chi_1,\chi_2)}{L(2s,\chi_{1,2})} \,ds.
\]
Li \cite[Theorem 2.2]{Li79} has shown that the Rankin--Selberg $L$-function $L_{RS}$ satisfies a functional equation which relates $L_{RS}(s;\chi_1,\chi_2)$ with $L_{RS}(1-s;\overline{\chi_1},\overline{\chi_2})$. Hereby we may deduce  
the upper bound 
\[
c(\chi_1,\chi_2) \ll_\eps 
\frac{(q_1q_2)^2}{(q_1,q_2)^{2-\eps}}
\] 
 for the conductor $c(\chi_1,\chi_2)$ of $L_{RS}(s;\chi_1,\chi_2)$ (see Remark \ref{rem:li-langlands}) and 
the Phragm\'en--Lindel\"of principle yields the convexity bound 
\begin{equation}\label{eq:bound-L_RS}
L_{RS}(1/2+\alpha+it;\chi_1,\chi_2) \ll_\eps  (q_1 q_2(1+|t|)^2)^{1/2-\alpha+\eps}
\end{equation}
for every $\eps>0$ 
 and all $t \in \Rr$.  
By the fast decay of $\widehat\phi$ and \eqref{eq:llll}, we thus get
\begin{equation}\label{eq:SN_versch}
S_N(\chi_1,\chi_2) \ll_\eps 
N^{1/2}(\log N)Q^{1+\eps}
\end{equation}
if $\chi_1 \neq \chi_2$.

\begin{remarkf}\label{rem:li-langlands}
The intricate general functional equation for Rankin--Selberg $L$-functions for convolutions of holomorphic cusp forms in {\cite[Theorem 2.2]{Li79}} simplifies considerably under our assumption that the level is a fundamental discriminant that is not an integral multiple of~$8$ -- at least, after working through the extensive notation that is necessary there (and noting that the definition of ``$N$'' in \cite[\S 2]{Li79} contains probably a typographical error as it should denote the \emph{least common multiple} and not the maximum of ``$N_1$'' and ``$N_2$''). For instance, the second and third product in \cite[(2.11)]{Li79} vanish and the conditions A)--C) on page 141 are trivially satisfied then.

The complexity of the functional equation in its general form 
displays the major drawback of considering these $L$-functions from the ``elementary'', 
classical viewpoint and not using the correspondence to \mbox{$L$-functions} of automorphic representations, which usually take a more natural form (see {\cite[\S2.3]{Mic07}} and the references there). 
The effort needed to apply this equation when $q_1,q_2$ are not fundamental discriminants seems disproportionate and one would certainly be well-advised to translate the situation to the automorphic setting then. 
Harcos and Michel \cite[p.~582]{HM06} mention that the bounds 
\[
\frac{(q_1 q_2)^2}{(q_1, q_2)^4} \leq c(\chi_1,\chi_2) \leq \frac{(q_1 q_2)^2}{(q_1, q_2)}
\]
 for the conductor of $L_{RS}(s;\chi_1,\chi_2)$ can be derived using the local Langlands correspondence, which then also yield the convexity bound \eqref{eq:bound-L_RS}. 
\end{remarkf}

\begin{remf}
Note that the existing subconvexity bounds for Rankin--Selberg convolutions either require that one of the two involved cusp forms is fixed \cite{HM06} or that one cusp form has a much smaller level than the other \cite{HM12arx}. 
 Although one may  
hope that more general results will be obtained in the 
future, these will probably only slightly improve our results (due to the saving of probably only a tiny power of the conductor) and will therefore be less important for us than for other applications.

The best bound one could hope for in \eqref{eq:bound-L_RS} is provided by the Lindel\"of Hypothesis. We will state the resulting large sieve inequality in Remark \ref{lem:lsi_complex_conditionalLinde}. 
\end{remf}

Let us come back to the proof of Lemma \ref{lem:lsi_complex}. If $\chi_1=\chi_2 \in \Hhatc(q)$, 
we use the bound
\begin{equation}\label{eq:lambda_chi_abschaetzen}
|\lambda_\chi(n)|\leq \sum_\stacksum{\aF \in Z(q)}{\No(\aF)=n} 1  \leq \prod_{p^v || n}(v+1) = \tau(n), 
\end{equation}
where the second inequality is due to the fact that each prime divisor $p$ of $n$ splits into at most two distinct prime ideals in the quadratic field $\Qppq$. 
Therefore
\begin{equation}\label{eq:SN_gleich}
S_N(\chi_1,\chi_1)\leq \sum_{n \geq 1} \tau(n)^2 \phi(n/N) \ll N(\log N)^3,
\end{equation}
where the implied constant is absolute (see \cite[(2.31)]{MV07}, for example).

\smallskip
Now that we have bounded $S_N(\chi_1,\chi_2)$ for all pairs $\chi_1, \chi_2 \in \Hhatc(Q)$, 
it remains to use  
a simple positivity argument 
and the duality principle in order to get the bound \eqref{eq:lsi_complex},   which we originally set out to prove: 
For all complex numbers $b_\chi$, indexed by the characters $\chi \in \Hhatc(Q)$, the positivity of~$\phi$ gives  
\begin{equation*}
\begin{aligned}
&\sum_{n \leq N} \Big| \sum_{\chi \in \Hhatc(Q)} b_\chi \lambda_\chi(n)\Big|^2 
\leq 
\sum_{\chi_1, \chi_2 \in \Hhatc(Q)} b_{\chi_1} \overline{b_{\chi_2}} S_N(\chi_1,\chi_2) 
\\
\leq \ & 
2 \Big(\max_{\chi_2 \in \Hhatc(Q)} \sum_{\chi_1 \in \Hhatc(Q)}|S_N(\chi_1,\chi_2)|\Big) \sum_{\chi_2 \in \Hhatc(Q)} |b_{\chi_2}|^2.
\end{aligned}
\end{equation*}
We insert the bounds \eqref{eq:SN_versch} and \eqref{eq:SN_gleich} 
 into the right-hand side of this inequality and note that
\begin{equation}\label{eq:MengeAllerCharaktereBound}
|\Hhatc(Q)| \leq \sum_{q \in \DF(Q)} h(q) \ll \sum_{q \in \DF(Q)} |q|^{1/2}(\log |q|) \ll Q^{3/2}(\log Q)
\end{equation}
by the upper class number bound 
$h(q) \ll |q|^{1/2} (\log |q|)$, 
which follows from the bound 
$L(1, \chi_q) \ll \log |q|$ (see \cite[Lemma 10.15]{MV07}, for example) and Dirichlet's class number formula. 
 Thus the bound 
\begin{equation*}
\sum_{n \leq N} \Big| \sum_{\chi \in \Hhatc(Q)} b_\chi \lambda_\chi(n)\Big|^2 \ll_\eps 
\left(N (\log N)^3+N^{1/2}(\log N)Q^{5/2+\eps}\right) 
 \sum_{\chi \in \Hhatc(Q)} |b_\chi|^2
\end{equation*}
holds for all tuples $(b_\chi)_{\chi \in \Hhatc(Q)}$ of complex numbers. By the duality principle (see \cite[p.~171]{IK04}, for example), this is equivalent to the statement of the lemma.
\end{proof}

\begin{remarkf}\label{lem:lsi_complex_conditionalLinde}
The Lindel\"of Hypothesis (for Rankin--Selberg convolutions of holomorphic cusp forms of weight one) yields 
\begin{equation*}\label{eq:L_RS_bound_on_LH}
L_{RS}(1/2+it;\chi_1,\chi_2)\ll_\eps (q_1q_2)^{\eps} (1+|t|)^{2\eps}.
\end{equation*}
This gives $S_N(\chi_1,\chi_2) \ll N^{1/2}(\log N)Q^{\eps}$ in \eqref{eq:SN_versch} and we therefore have the following conditional large sieve inequality:
\begin{equation*}
\sum_{\chi \in \Hhatc(Q)} \Big| \sum_{n \leq N} a_n \lambda_\chi(n)\Big|^2 \ll_\eps 
\left(N (\log N)^3+N^{1/2}(\log N)Q^{3/2+\eps}\right) \sum_{n \leq N} |a_n|^2
\end{equation*}
for all complex numbers $(a_n)_{n \leq N}$ and all $\eps>0$, $Q \geq 1$ and $N \geq 3$.

Given the fact that the essentially best-possible large sieve inequality for Dirichlet characters, Lemma \ref{lem:lsi_org}, can be proved unconditionally, there is some reason to hope that it might be possible to improve Lemma \ref{lem:lsi_complex} without employing any kind of subconvexity bounds for the involved $L$-functions.
\end{remarkf}

In the proof of our variant of the Bombieri--Vinogradov theorem we will need the large sieve inequality for complex class group characters in the following form: 
\begin{corollary}\label{corr:lsi_02}
Let $(a_n)$ be a complex sequence with $\sum_{n\geq 1} |a_n| < \infty$. Let $Q \geq 1$, $k \geq 2$, $c \geq \frac 1 2$  
 and $\eps>0$. Then
\begin{align}
&\sum_{\chi \in \Hhatc(Q)}  \int_{(c)} \bigg|\sum_{n\geq 1} a_n \lambda_\chi(n) n^{-s}\bigg|^2 |s|^{-(k+1)}\, |ds| \nonumber
\\
\ll_\eps & \ Q^{3/2+\eps} \sum_{n \leq Q^2} |a_n|^2n^{1-2c} (1+(\log n)^3) \label{eq:LSI_int_uncond_with_trivial}
+ Q^{\eps} \sum_{n > Q^2} |a_n|^2(n^{1-2c}+n^{1/2-2c}Q^{5/2}) (\log n)^3
\\
\ll_\eps & \ 
Q^{\eps} \sum_{n \geq 1} |a_n|^2(n^{1-2c}+n^{1/2-2c}Q^{5/2}) (1+(\log n)^3).\label{eq:LSI_int_uncond}
\end{align}
Moreover, we have
\begin{equation}\label{eq:LSI_int_cond}
\sum_{\chi \in \Hhatc(Q)}  \int_{(c)} \bigg|\sum_{n\geq 1} a_n \lambda_\chi(n) n^{-s}\bigg|^2 |s|^{-(k+1)}\, |ds| \ll_\eps  Q^{\eps} \sum_{n \geq 1} |a_n|^2(n^{1-2c}+n^{1/2-2c}Q^{3/2}) (1+(\log n)^3)
\end{equation}
if the Lindel\"of Hypothesis holds.
\end{corollary}
\begin{proof}
The bounds \eqref{eq:LSI_int_uncond} and \eqref{eq:LSI_int_cond} 
follow from Lemma \ref{lem:lsi_complex} and Remark \ref{lem:lsi_complex_conditionalLinde}, respectively, along the lines of the proofs of \cite[{Th\'eor\`eme 10}]{Bom87} and \cite[Corollary 3.3]{MP13}. As for the bound \eqref{eq:LSI_int_uncond_with_trivial}, we additionally note that if $N \leq Q^2$, then the trivial bound
\[
\sum_{\chi \in \Hhatc(Q)} \Big| \sum_{n \leq N} a_n \lambda_\chi(n)\Big|^2 \ll_\eps Q^{3/2+\eps} N (\log N)^3 \sum_{n \leq N} |a_n|^2,
\]
which follows for all $\eps>0$ from the Cauchy--Schwarz inequality, \eqref{eq:lambda_chi_abschaetzen} and \eqref{eq:MengeAllerCharaktereBound}, is at least as good as the bound 
in Lemma \ref{lem:lsi_complex}. 
\end{proof}


\section{Smooth results of Bombieri--Vinogradov type}\label{sec:bv}

Being now equipped with the basic notions and a large sieve inequality for complex class group characters, we may now proceed to the proof of Theorem~\ref{thm:BV_01_PNT}. We will derive it from a ``well-distribution'' result for smoothed versions of a Chebyshev-type function for integers represented by binary quadratic forms. 
Interestingly, we may save here a positive power of $X$ over ``trivial'' bounds if we confine ourselves to sets $M(Q) \subseteq \DF(Q)$ of discriminants for which no (positive or negative) fundamental discriminant has many integer multiples in~$M(Q)$ (see Remark \ref{rem:save_power_of_X}).

\begin{definition}\label{def:Bedingung_Ad}
	For any $Q \geq 1$, let $M(Q)$ be a subset of $\DF(Q)$. 
	We say that 
	$\nu \in[0,1]$ is a \emph{divisor frequency} of $M(Q)$ 
	if it satisfies the property:
  \begin{gather}\label{Bedingung_M} 
  \begin{aligned}
  &\text{The cardinality of the set $\{q \in M(Q):\, q'  \mid q\}$ is 
	at most $Q^{\nu}$ for each}\\&\text{(positive or negative) fundamental discriminant $q'$ with $1 < |q'| \leq Q$.}
  \end{aligned}
  \end{gather}
\end{definition}

\noindent For all $X \geq 3$, all $q \in \DF$, all $C \in \Kcl(q)$ and all integers $k \geq 0$, we define
   \begin{align}\label{eq:def_psik}
	\psi_k(X;q,C)=\ &\frac{1}{k!}\sum_\stacksum{n\leq X}{n \in \Repr(q,C)} \Lambda(n) \left(\log \frac{X}{n}\right)^k w(C,n),
   \end{align}
	where $w(C,n)$ is given by \eqref{eq:weight_w-Cn_for_primepowers}.

\begin{theorem}\label{thm:BV_01}
	Let $M(Q) \subseteq \DF(Q)$ for some $Q \geq 1$ and let $\nu \in (0,1]$ be a divisor frequency of $M(Q)$. For every integer $k \geq 2$, every (arbitrarily large) real number $A>0$ and every (arbitrarily small) real number $\eps>0$, there exists a real number $B=B(A)$ such that
  \begin{equation}\label{eq:my_bv_thm}
  \sum_{q \in M(Q)} \max_{C \in \Kcl(q)} \max_{Y \leq X} \bigg|\psi_k(Y;q,C)- \frac{1}{h(q)}\sum_{K \in \Kcl(q)}\psi_k(Y;q,K) \bigg|\ll Q^{\nu/2}X(\log X)^{-A}
  \end{equation}
  for $ Q^{4(1+(2-\nu)(3-\nu)/3)+\eps} \leq X(\log X)^{-B}$. 
	The implied constant depends on $\eps$, $A$, $k$ and $\nu$; the dependence on $\eps$ is effective, the dependence on $A$, $k$ and $\nu$ is non-effective. The constant $B$ is explicitly computable; in particular, one may choose $B=16A+300$.
\end{theorem}

If the set $M(Q)$ is composed of negative prime discriminants, 
then $\nu=0$ is a divisor frequency of $M(Q)$. 
In this case we just fail to achieve \eqref{eq:my_bv_thm} with $\nu=0$. Nevertheless, it is worth recording that the proof of Theorem \ref{thm:BV_01} yields: 
\begin{theorem}\label{thm:BV_prime}
Let $Q \geq 1$ and let $\Pi(Q)$ be the set of negative \emph{prime} discriminants whose absolute value is at most $Q$. 
For every integer $k \geq 2$ 
 and every (arbitrarily small) real number $\eps>0$, 
we may find an absolute constant $B$ such that   
\begin{equation}\label{eq:bv_primediscr1}
  \sum_{q \in \Pi(Q)} \max_{C \in \Kcl(q)} \max_{Y \leq X} \bigg|\psi_k(Y;q,C)- \frac{1}{h(q)}\sum_{K \in \Kcl(q)}\psi_k(Y;q,K)\bigg|\ll_{\eps,k} X(\log X)^{k+3}
  \end{equation}
	for $ Q^{12+\eps} \leq X(\log X)^{-B}$. 
\end{theorem}

\begin{remarkf}\label{rem:save_power_of_X}
To put this last result into perspective, set $f_{q}(x,y)=x^2+xy+\frac{1-q}{4}y^2$, say, for each negative fundamental prime discriminant $q \equiv 1 \mods 4$ and  consider the function 
\[
S_{q}(X)=\sum_\stacksum{p\leq X}{\exists x,y \in \Zz:\, f_{q}(x,y)=p} 
\log(p) \left(\log \frac{X}{p}\right)^2,
\]
which gives a smoothed and weighted count of the primes up to $X$ that can be represented by the form $f_{q}$ (which lies in the principal class $C_0$ of discriminant $q$). By \eqref{eq:weight_w-Cn_on_average} and Theorem~\ref{thm:cox77}, we have 
\[
S_{q}(X)
=\demi \sum_\stacksum{n\leq X}{\exists x,y \in \Zz:\, f_{q}(x,y)=n} 
\Lambda(n) \left(\log \frac{X}{n}\right)^2 w(C_0,n) +O(X^{1/2}(\log X)^3)
\]
for negative fundamental discriminants $q$ with $q \equiv 1 \mods 4$ and $|q| \leq X$.  Thus, Theorem~\ref{thm:BV_prime} implies that, for most negative prime discriminants $q$ with $|q| \leq X^{1/13}$, the function $S_{q}(X)$
 deviates from the (expectable) average function 
\[
\frac{1}{2h(q)}\sum_{K \in \Kcl(q)}e(K)\sum_\stacksum{p\leq X}{p \in \Repr(q,K)} 
\log(p) \left(\log \frac{X}{p}\right)^2 
\]
 by only a small amount at most -- 
and the sum (over $q \in \Pi(Q)$) of these discrepancies is a positive power of $X$ smaller than ``trivial'' estimates can guarantee. 
Indeed, if $X$ is large, $Q=X^{1/13}$ and $k=2$, then Theorem \ref{thm:BV_prime} beats the easy bound (compare Remark \ref{rem:BV_comparison_to_trivial}) 
\begin{equation*}
O\bigg( X (\log X)^{3} \cdot \Big(\frac{Q}{\log Q}\Big)^{1/2} \bigg)
\end{equation*}
for the left-hand side of \eqref{eq:bv_primediscr1} by a factor of size 
$
\Big(\frac{Q}{\log Q}\Big)^{1/2} (\log X)^{3-5} \gg_\eps X^{1/26-\eps}
$ 
 for all arbitrarily small $\eps>0$. 
This result is unusual as it does not seem to be possible to achieve a saving of a positive power of $X$ over the trivial bound for the corresponding smooth version of the original Bombieri--Vinogradov theorem.
\end{remarkf}

\begin{remarkf}\label{rem:BV_Linde2}
Under the assumption of the Lindel\"of Hypothesis, 
Theorem \ref{thm:BV_01} holds with 
$Q^{4-2\nu+\eps} \leq X (\log X)^{-B}$ if $\nu \geq \demi$ and with $Q^{4(2-\nu)^2/3+\eps} \leq X (\log X)^{-B}$ if $\nu  < \demi$.  
Theorem \ref{thm:BV_prime}  then 
holds with $ Q^{16/3+\eps} \leq X(\log X)^{-B}$; see Remark \ref{rem:BV-Linde01}.
\end{remarkf}

\begin{remf}
If $\nu<1$, then it does not seem to be possible to unsmooth these results, i.e.\ to take $k=0$, while keeping the given estimates, because the unsmoothing process 
produces a term of size $Q^{1/2}X(\log X)^{-D} $ (where $D$ is an arbitrary positive number). 
\end{remf} 

However, for $\nu=1$, i.e.\ for arbitrary sets $M(Q) \subseteq \DF(Q)$ of negative fundamental discriminants, 
these extra terms of size $Q^{1/2}X(\log X)^{-D} $ are not too large and we obtain:
\begin{theorem}\label{thm:BV_01_unsmoothed}
For all $q \in \DF$ and all $C \in \Kcl(q)$, define
\[
\psi(X;q,C) = \sum_\stacksum{n\leq X}{n \in \Repr(q,C)} \Lambda(n). 
\]
Let $A>0$ 
and $\eps>0$. 
Let $e(C)$ be defined by \eqref{eq:definition_of_eK}. Then there exists $B=B(A)$ such that
  \begin{equation*}
  \sum_{q \in \DF(Q)} \max_{C \in \Kcl(q)} \max_{Y \leq X} \left|\psi(Y;q,C)- \frac{Y}{e(C) h(q)}\right|\ll_{\eps,A} Q^{1/2}X(\log X)^{-A}
  \end{equation*}
  for $ Q^{20/3+\eps} \leq X(\log X)^{-B}$. 
	The constant $B$ is explicitly computable; in particular, one may choose $B=64A+350$.
\end{theorem}
\noindent As usual, Theorem \ref{thm:BV_01_PNT} follows by partial integration from this result.


\section{Proofs of the Bombieri--Vinogradov type results}\label{sec:prelim_1}

Let $A>0$ (arbitrarily large) and $\eps>0$ (arbitrarily small) be real numbers; let $k \geq 2$ be an integer; let $M(Q) \subseteq \DF(Q)$ be a set of negative fundamental discriminants $q \not \equiv 0 \mods 8$ 
 with divisor frequency $\nu \in [0,1]$. These numbers will be considered as fixed parameters which the implied constants in the estimates of this and the subsequent two sections may depend on. 

Let $X \geq Q$. By definitions \eqref{eq:weight_w-Cn_for_primepowers} and \eqref{eq:def_psik}, we have
\[
\psi_k(X;q,C)=\frac{1}{k!}\sum_\stacksum{\aF \in B_q(C) \cap Z(q)}{\No(\aF) \leq X}\Lambda(\No(\aF)) \left(\log \frac{X}{\No(\aF)}\right)^k
\]
for all $q \in \DF(Q)$ and all $C \in \Kcl(q)$. For ease of notation we set
\begin{equation}\label{eq:def_Ek}
E_k(X;q)=\max_{C \in \Kcl(q)} \max_{Y \leq X} \bigg|\psi_k(Y;q,C)- \frac{1}{h(q)}\sum_{K \in \Kcl(q)}\psi_k(Y;q,K)\bigg|.
\end{equation}
Thus, if the assumptions of Theorem \ref{thm:BV_01} and Theorem \ref{thm:BV_prime} hold, we have to prove the bounds 
\begin{equation}\label{eq:my_bv_thm_ideal}
  \sum_{q \in M(Q)}  E_k(X;q) \ll Q^{\nu/2}X(\log X)^{-A} 
  \end{equation}
if $\nu>0$ and $Q^{4(1+(2-\nu)(3-\nu)/3)+\eps} \leq X(\log X)^{-B(A)}$, 
and 
\begin{equation}\label{eq:bv_primediscr1_ideal}
  \sum_{q \in \Pi(Q)} E_k(X;q) \ll X(\log X)^{k+3}  
  \end{equation}
if $Q^{12+\eps} \leq X(\log X)^{-B}$.

\medskip
We start the proof of both \eqref{eq:my_bv_thm_ideal} and \eqref{eq:bv_primediscr1_ideal} by appeal to the orthogonality property of the finite abelian groups $\Hhat(q)$ of ideal class group characters. Define
\[
\psi_k(Y;q,\chi)=\frac{1}{k!}\sum_\stacksum{\aF \in Z(q)}{\No(\aF) \leq Y}\Lambda(\No(\aF)) \chi(\aF) \left(\log \frac{Y}{\No(\aF)}\right)^k
\]
for all $q \in \DF(Q)$, all $\chi \in \Hhat(q)$ and all $k \geq 0$.
Orthogonality  
yields
\begin{equation*}
\begin{aligned}
\psi_k(Y;q,C)=&\ \sum_\stacksum{\aF \in Z(q)}{\No(\aF)\leq Y}\Lambda(\No(\bF))\bigg(\log \frac{X}{\No(\aF)}\bigg)^k \Bigg( \frac{1}{h(q)}\sum_{\chi \in \Hhat(q)} \overline{\chi}(B_q(C)) \chi(\aF)\Bigg)
\\
=&\ \frac{1}{h(q)}\sum_{\chi \in \Hhat(q)} \overline{\chi}(B_q(C))\psi_k(Y;q,\chi)
\end{aligned}
\end{equation*}
for all $q \in \DF(Q)$ and all $C \in \Kcl(q)$. Together with the triangle inequality we thus get
\begin{equation}\label{eq_BV1}
\begin{aligned}
\sum_{q\in M(Q)} E_k(X;q) \leq \ & \max_{Y \leq X}  \sum_{q\in M(Q)}  \frac{1}{h(q)}\sum_{\chi\neq \chiq_0}\left|\psi_k(Y;q,\chi)\right|.  
\end{aligned}
\end{equation}

As before, for every 
$q \in \DF$, we let $\chi_q$ denote the unique primitive real Dirichlet character modulo $|q|$. 
By Siegel's theorem (see \cite[Theorem 11.14]{MV07}, for example), we have the unconditional, non-effective lower bound $|q|^{-\eps}\ll_\eps L(1,\chi_q)$ for the corresponding Dirichlet $L$-function. This yields the lower class number bound 
\begin{equation*}\label{eq:classno_bound_fuer_alle}
|q|^{1/2-\eps}\ll_\eps h(q)
\end{equation*}
by Dirichlet's class number formula (see \cite[(2.31)]{IK04}, for example). Yet, there exists a better bound for many $q$ and it turns out that the contribution from the other discriminants is often negligible: 
		We know 
	(see \cite[Theorem 11.3]{MV07}) 
	that there exists an absolute constant $c_1>0$ such that, for any $q \in \DF$,  the Dirichlet $L$-function $L(s,\chi_q)$ has at most one zero, the \emph{Landau--Siegel zero} for the modulus $|q|$, in the set 
  \begin{equation*}
	\Big\{s=\sigma+it \in \Cc:\,  \sigma \geq 1 -\frac{c_1}{\log{|q|(|t|+4)}} \Big\}.
  \end{equation*} 
  Moreover, there exists $c_2=c_2(c_1)>0$ such that 
	 $L(1, \chi_q) \geq c_2 (\log |q|)^{-1}$ 
	if $L(s,\chi_q)$ has no Landau--Siegel zero (see \cite[Theorem 11.4]{MV07}). 
	Thus, by the class number formula, 
	there exists $c_3=c_3(c_1)>0$ such that
 \begin{equation}\label{eq:classno_bound1}
|q|^{1/2}(\log |q|)^{-1}\leq c_3 h(q)
\end{equation} 
holds for all $q \in \DF$ for which $L(s,\chi_q)$ has no Landau--Siegel zero. We fix such a value of $c_3$.

The following proposition will give an upper bound for the contribution to 
the right side of \eqref{eq_BV1} coming from the (presumably empty) set $\DF_{\textup{ex}}(Q) \subset \DF(Q)$ of exceptional fundamental discriminants; here we call $q \in \DF$ exceptional if it fails to satisfy \eqref{eq:classno_bound1} for the fixed value of $c_3$ (and therefore $L(s,\chi_q)$ has a Landau--Siegel zero then). 
\begin{proposition}\label{prop:landau-siegel1}
Let $M_{\textup{ex}}(Q)=\DF_{\textup{ex}}(Q) \cap M(Q)$ be the (possibly empty) subset of exceptional fundamental discriminants of $M(Q)$.
Then we have
  \[
  \max_{Y \leq X} \sum_{q\in M_{\textup{ex}}(Q)}\frac{1}{h(q)}  \sum_{\chi \in \Hhat(q) \smallsetminus\{\chiq_0\}}  |\psi_k(Y;q,\chi)| \ll (\log Q) X (\log X)^{k+2}.
  \]
	In particular, exceptional discriminants contribute negligibly to 
	the right side of \eqref{eq_BV1} if either $\nu>0$ and $Q \geq (\log X)^{(2A+2k+6)/\nu}$ or $\nu=0$. 
\end{proposition}

\begin{remarkf}\label{rem:SW+klein}
The case $Q < (\log X)^{(2A+2k+6)/\nu}$ will be dealt with later on by means of an appropriate Siegel--Walfisz type theorem; see Remark \ref{rem:bv_small_discr} below. 
Moreover, note that if $\nu=0$, then this contribution would not be negligible in Theorem \ref{thm:BV_01}, which is why we get the slightly weaker bound in Theorem \ref{thm:BV_prime}.
\end{remarkf}

\begin{proof}
Let $q_1$ be an exceptional modulus. 
By a theorem of Landau (see \cite[Corollary 11.9]{MV07}), we know that there cannot exist an exceptional modulus $q$ with $q_1 < q< q_1^2$. Thus, there can be at most $\frac{\log Q}{\log 2}$ exceptional moduli which are smaller than $Q$.  Using 
standard estimates (see \eqref{eq:lambda_chi_abschaetzen}), we also have 
\begin{equation*}
\big|\psi_k(Y;q,\chi)\big| = \Big| \sum_{n \leq Y} \Lambda(n)  \left(\log \frac Y n \right)^k \sum_\stacksum{\mathfrak b \in Z(q)}{\No(\bF)=n}\chi(\bF) \Big|
\leq
 (\log X)^k\sum_{n \leq X} \log(n) \tau(n) \ll X(\log X)^{2+k}
\end{equation*}
for all $q \in M_{\textup{ex}}(Q)$ and all $\chi \in \Hhat(q)$, 
and the first assertion follows immediately. 

If $\nu>0$ and $Q \geq (\log X)^{(2A+2k+6)/\nu}$, then 
\[
(\log Q) X (\log X)^{2+k} \leq Q^{\nu/2} X (\log X)^{-A},
\]
i.e.\ the contribution from exceptional discriminants is acceptable for Theorem \ref{thm:BV_01}.
\end{proof} 

Therefore it remains to estimate the contribution from non-exceptional discriminants 
on the right side of \eqref{eq_BV1}, i.e.\ we have to bound
\begin{equation}\label{eq:bv_ungesplittet}
\max_{Y \leq X}  \sum_{q\in M'(Q)} \frac{1}{h(q)}\sum_{\chi\neq \chiq_0}\left|\psi_k(Y;q,\chi)\right|,
\end{equation}
where 
\[
M'(Q)=M(Q) \smallsetminus M_{\text{ex}}(Q)
\]
or
\[
M'(Q)=\Pi(Q) \smallsetminus M_{\text{ex}}(Q),
\]
and we will show that it is bounded above by 
\begin{equation}\label{eq:bv_ungesplittet_desired_1}
Q^{\nu/2}X(\log X)^{-A}
\end{equation}
for both $\nu>0$ and $\nu=0$.

\medskip
If $Q$ is very small, a uniform bound for $\psi_0(X;q,\chi)$ exists, which easily yields this desired bound for 
\eqref{eq:bv_ungesplittet};  
the following statement is a special case of Goldstein's generalization of the Siegel--Walfisz theorem \cite{Gol70}:
\begin{proposition}[Goldstein]\label{prop:sw-goldstein}
Suppose that $q \in \DF$ with $|q| \leq (\log X)^D$ for some positive constant~$D$. Then
\begin{equation*}\label{eq:SW_Goldstein}
\psi_0(X;q,\chi) \ll_D X (\log X)^{-2D}
\end{equation*} 
 for all  non-trivial class group characters $\chi \in \Hhat(q)$. 
The implied constant does not depend on $q$ or $\chi$, but 
 is ineffective.
\end{proposition}
\noindent So suppose that $Q=(\log X)^D$ for some $D \geq A+k$. 
We have
\begin{equation*}\label{eq:psi_k_runter}
\psi_k(Y;q,\chi)=\int_{1}^{Y} \psi_{k-1}(t;q,\chi) \frac{dt}{t} \ll \max_{y \leq Y}|\psi_{0}(y;q,\chi)|\cdot (\log Y)^{k}.
\end{equation*}
 Summing over $q \in M'((\log X)^D)$, Proposition \ref{prop:sw-goldstein} therefore yields the upper bound \eqref{eq:bv_ungesplittet_desired_1}  
for~\eqref{eq:bv_ungesplittet} if $Q=(\log X)^D$. 

\begin{remarkf}\label{rem:bv_small_discr}
We 
have now proved that the bounds in both Theorem \ref{thm:BV_01} and Theorem~\ref{thm:BV_prime} hold for $Q \leq (\log X)^D=:Q_0$ and it remains to bound \eqref{eq:bv_ungesplittet} with $M'(Q)$ replaced by 
\[
M''(Q):=M'(Q) \cap \{q:\, |q|> Q_0\}
\] for a value of $D$ that we will choose in the next section (see \eqref{eq:setze_log_power}). We already record that, because of   Remark \ref{rem:SW+klein}, we must choose $D$ at least as large as $D_1:=(2A+2k+6)/\nu$ if $\nu>0$.  If $\nu =0$, we will have to  choose some $D \geq D_1:=A+k$ to guarantee the bound \eqref{eq:bv_ungesplittet_desired_1}  for \eqref{eq:bv_ungesplittet} (which is more than enough for Theorem \ref{thm:BV_prime}).
\end{remarkf}

\smallskip
The {class group $L$-functions}, i.e.\ the $L$-functions associated to the characters $\chi \in \Hhat(q)$ for each $q \in \DF$, are given by 
\[
L(s,\lambda_\chi):=\sum_{\aF \in Z(q)} \frac{\chi(\aF)}{\No(\aF)^s}=\sum_{n\geq 1}\frac{\lambda_\chi(n)}{n^s}
\]
for $\Reel(s)>1$, 
where $\lambda_\chi(n)$ is defined by \eqref{eq:def_lambda}. 
Each of these series has an analytic continuation to the whole complex plane unless $\chi=\chiq_0$ when the continuation is meromorphic with a pole at $s=1$ 
(see \cite[\S 7]{Nar04}, for example). 
The expansion of the logarithmic derivative  of such an $L$-function is given by 
\[
\frac{L'}{L}(s,\lambda_{\chi})=- \sum_{\aF \in Z(q)} \widetilde{\Lambda}(\aF) \chi(\aF) {\No(\aF)^{-s}},
\]
where
\begin{equation*}\label{eq:vonMangoldt_forfields}
\widetilde{\Lambda}(\aF)= 
\begin{cases}
\log \No(\pF) \ &\text{if $\aF=\pF^m$ for some prime ideal $\pF \in Z(q)$ and some integer $m$},\\
0 \ &\text{otherwise}.
\end{cases}
\end{equation*}
Thus, the $k$-th iteration of the inverse Mellin transform of $\frac{L'}{L}$ is (see \cite[(5.22)]{MV07}, for example)
\begin{equation*}\label{eq:mellin1}
-\frac{1}{2\pi i}\int_{(c)}\frac{L'}{L}(s,\lambda_{\chi})Y^s s^{-(k+1)}\, ds=\frac{1}{k!}\sum_\stacksum{\aF \in Z(q)}{\No(\aF) \leq Y} \widetilde{\Lambda}(\aF) \chi(\aF) \left(\log \frac{Y}{\No(\aF)}\right)^k=:\wpsi_k(Y;q,\chi)
\end{equation*}
for all $c>1$. This does not equal $\psi_k(Y;q,\chi)$, but we miss it only by a negligible margin: 
Set 
\[
c(\aF)=\chi(\aF) \left(\log \frac{Y}{\No(\aF)}\right)^k
\]
and note that we have
\begin{align*}
k!\, \wpsi_k(Y;q,\chi) 
=&\ \sum_{p \leq Y} \sum_\stacksum{\pF \in Z(q)}{\No(\pF)=p} (\log p) c(\pF)   
+\sum_{p \leq Y^{1/2}} \sum_\stacksum{\pF \in Z(q)}{\No(\pF)=p^2}  (\log p^2) c(\pF)    
+\sum_{\ell\geq 2}\sum_\stacksum{\pF \in Z(q)}{\No(\pF)^\ell \leq Y} \log(\No(\pF)) c(\pF^\ell)\\
=&\ \sum_{p \leq Y} \sum_\stacksum{\pF \in Z(q)}{\No(\pF)=p} (\log p) c(\pF)  + O(Y^{1/2} (\log Y)^{k+3})   
\end{align*}
and
\begin{align*}
k!\, \psi_k(Y;q,\chi)
=&\ \sum_{\ell \geq 1} \sum_{p \leq Y^{1/\ell}} \sum_\stacksum{\aF \in Z(q)}{\No(\aF)=p^\ell} (\log p) c(\aF)\\
=&\ \sum_{p \leq Y} \sum_\stacksum{\pF \in Z(q)}{\No(\pF)=p} (\log p) c(\pF)+ 
\sum_{\ell\geq 2}\sum_{p \leq Y^{1/\ell}}  (\log p) \sum_\stacksum{\aF \in Z(q)}{\No(\aF)=p^\ell} c(\aF)\\
=&\ \sum_{p \leq Y} \sum_\stacksum{\pF \in Z(q)}{\No(\pF)=p} (\log p) c(\pF)+ O(Y^{1/2}(\log Y)^{k+3}).
\end{align*}
Hence
\begin{equation*}\label{eq:mellin_wpsi}
\psi_k(Y;q,\chi) = \wpsi_k(Y;q,\chi) + O(Y^{1/2}(\log Y)^{k+3}).
\end{equation*}
Summing over $q \in M''(Q)$, the contribution of the remainder terms is $\ll QX^{1/2}(\log X)^{k+3}$ in~\eqref{eq:bv_ungesplittet} if we replace $\psi_k(Y;q,\chi)$ by 
$\wpsi_k(Y;q,\chi)$ there.  But this is negligible in 
\eqref{eq:my_bv_thm_ideal} and \eqref{eq:bv_primediscr1_ideal}. Thus it remains to estimate
\begin{equation}\label{eq:bv_ungesplittet2}
\max_{Y \leq X} \sum_{q \in M''(Q)} \frac{1}{h(q)} \sum_\stacksum{\chi \in \Hhat(q)} {\chi \neq \chiq_0} |\wpsi_k(Y;q,\chi)|.
\end{equation}
Next, we split \eqref{eq:bv_ungesplittet2} into 
\begin{equation}\label{eq:bv_gesplittet2}
\begin{aligned}
&\max_{Y \leq X} \sum_{q \in M''(Q)} \frac{1}{h(q)} \sum_\stacksum{\chi \in \Hhat(q)} {\chi^2 \neq \chiq_0} |\wpsi_k(Y;q,\chi)|
 \ +\  \max_{Y \leq X} \sum_{q \in M''(Q)} \frac{1}{h(q)} \sum_\stacksum{\chi \in \Hhat(q) \smallsetminus\{\chiq_0\}}{ \chi^2 =\chiq_0} |\wpsi_k(Y;q,\chi)| 
\\=&\ 
E'_k(Q,X) + E''_k(Q,X), 
\end{aligned}
\end{equation}
say, i.e.\ we split it into sums over complex class group characters and sums over real class group characters. 
We will estimate both terms 
separately in the next two sections and show that they are both bounded above by \eqref{eq:bv_ungesplittet_desired_1}:
 
In Section \ref{sec:complexchar1}, we show that $E'_k(Q,X)$ is of the desired size if
\begin{equation}\label{eq:Q_range_final_for_bv1}
Q^{4(1+(2-\nu)(3-\nu)/3)+\eps} \leq X (\log X)^{-B}; 
\end{equation} 
moreover, we may choose 
\begin{equation}\label{eq:B_final_for_bv1}
B=16A+300.
\end{equation} 

In Section \ref{sec:realchar1}, we show that $E''_k(Q,X)$ is of the desired size if
\[
Q^{5-3\nu} \leq X(\log X)^{-B}
\]
and we may choose $B=6A+40$. Since this range is larger than \eqref{eq:Q_range_final_for_bv1} and this value of $B$ is smaller than \eqref{eq:B_final_for_bv1}, the final admissible range and the final admissible value of $B$ for Theorems~\ref{thm:BV_01} and \ref{thm:BV_prime} are given by \eqref{eq:Q_range_final_for_bv1} and \eqref{eq:B_final_for_bv1}, respectively.
 
Together with the results for exceptional discriminants (Proposition \ref{prop:landau-siegel1}) and small discriminants  (Remark \ref{rem:bv_small_discr}) 
we may then conclude that \eqref{eq:my_bv_thm_ideal} and \eqref{eq:bv_primediscr1_ideal} hold. This finishes the proofs of Theorems~\ref{thm:BV_01} and \ref{thm:BV_prime}.

\medskip
As for the proof of Theorem \ref{thm:BV_01_unsmoothed}, we start by recalling that \eqref{eq:weight_w-Cn_for_primepowers} yields $w(C,p^\ell) \leq \ell +1$ 
 for all form classes $C$, all primes $p$ and all positive integers $\ell$. 
Moreover, \eqref{eq:weight_w-Cn_for_primepowers} and \eqref{eq:definition_of_eK} 
also yield
\begin{equation*}\label{eq:wKp_eingesetzt}
\sum_\stacksum{p\leq Y}{p \in \Repr(q,C)} (\log p) w(C,p) = e(C) \sum_\stacksum{p\leq Y}{p \in \Repr(q,C)} (\log p)+ O((\log Y)(\log|q|)).
\end{equation*}
Thus, for all $q \in \DF(Q)$, all $C \in \Kcl(q)$ and all $Y\leq X$, we have 
\begin{equation}\label{eq:Transition_from_psi0_to_psi}
\Big|\psi(Y;q,C)- \frac{Y}{e(C)h(q)} \Big|
\leq 
\Big|\psi_0(Y;q,C) - \frac{Y}{h(q)} \Big|+ O(Y^{1/2} (\log Y)^{3}+(\log Y)(\log|q|)).
\end{equation}
Summing over $q \in \DF(Q)$, we see that the remainder term is negligible in Theorem~\ref{thm:BV_01_unsmoothed}. 

Similar to the argument in \cite[\S 7.4]{Bom87}), one may easily show that
\begin{equation}\label{eq_bound_for_psi0_01}
\sum_{ q\in \DF(Q)} \max_{C \in \Kcl(q)} \max_{Y \leq X} \Big|\psi_0(Y;q,C)- \frac{Y}{h(q)} \Big| \ll Q^{1/2}X (\log X)^{-(A'-3)/4}
\end{equation}
holds if 
\begin{equation}\label{eq_bound_for_psi2_01}
\sum_{ q\in \DF(Q)} \max_{C \in \Kcl(q)} \max_{Y \leq X} \Big|\psi_2(Y;q,C)- \frac{Y}{h(q)} \Big| \ll Q^{1/2}X (\log X)^{-A'}
\end{equation}
holds for some $A'>0$. 
Therefore, Theorem \ref{thm:BV_01_unsmoothed} will follow from \eqref{eq:Transition_from_psi0_to_psi} and \eqref{eq_bound_for_psi0_01} as soon as we prove the bound \eqref{eq_bound_for_psi2_01} for 
\begin{equation}\label{eq:range_zum_letzten_Mal_fuer_BV}
Q^{20/3+\eps} \leq X(\log X)^{-B}
\end{equation}
 with 
$B=B(A')=16A'+300$ and then set $A'=4A+3$. 

 We split the left side of \eqref{eq_bound_for_psi2_01} into 
\begin{equation}\label{eq:Transition_from_psi2_to_psi2}
\begin{aligned}
&\sum_{ q\in \DF(Q)} \max_{C \in \Kcl(q)} \max_{Y \leq X}\big|\psi_2(Y;q,C)- \frac{Y}{h(q)} \big|
\\
\leq &\ 
\sum_{ q\in \DF(Q)} \max_{C \in \Kcl(q)} \max_{Y \leq X}
\big|\psi_2(Y;q,C)- \frac{1}{h(q)}\sum_{K \in \Kcl(q)}\psi_2(Y;q,K) \big|
\\
&\ + \ 
\sum_{ q\in \DF(Q)}  \max_{Y \leq X}
 \frac{\big|Y - \sum_{K \in \Kcl(q)}\psi_2(Y;q,K) \big|}{h(q)} .
\end{aligned}
\end{equation}
The first term on the right side of \eqref{eq:Transition_from_psi2_to_psi2} is $\ll Q^{1/2}X (\log X)^{-A'}$ by Theorem \ref{thm:BV_01} if 
\eqref{eq:range_zum_letzten_Mal_fuer_BV} holds and $B=16A'+300$. As for the second term, we note that equation \eqref{eq:weight_w-Cn_on_average} yields
\[
\sum_{K \in \Kcl(q)} \sum_\stacksum{p\leq Y}{p \in \Repr(q,K)} (\log p) \left(\log \frac{Y}{p}\right)^2 w(K,p) = \sum_{p\leq Y} (\log p) \left(\log \frac{Y}{p}\right)^2 (1+ \chi_q(p)). 
\]
Thus
\begin{equation*}
\begin{aligned}
& \ \big|Y - \sum_{K \in \Kcl(q)}\psi_2(Y;q,K) \big| 
\\
\leq & \ 
 \Big|Y - \demi\sum_{K \in \Kcl(q)} \sum_\stacksum{p\leq Y}{p\in \Repr(q,K)} (\log p) \left(\log \frac{Y}{p}\right)^2 w(K,p) \Big| 
\ +\ O\big(Y^{1/2}(\log Y)^3\big)
\\
\leq & \ 
 \Big(\big|Y - \psi_2(Y) \big|+ \big|\psi_2(Y;\chi_q)  \big| \Big)
+ O\big(Y^{1/2}(\log Y)^3\big), 
\end{aligned}
\end{equation*}
where 
\begin{equation}\label{eq:psik_for_Dirichlet}
{\psi}_k(Y;\chi_{q}) =\frac{1}{k!}\sum_{n \leq Y}\chi_q(n)\Lambda(n) \left(\log \frac{Y}{n}\right)^k
\end{equation} 
for each fundamental discriminant $q\neq 1$ and $\psi_2(Y):=\psi_2(Y;1)$. Summing over $q \in \DF(Q)$, we see that the remainder term is negligible in Theorem \ref{thm:BV_01_unsmoothed}. By the 
relation
\begin{equation}\label{eq:von_k+1_zu_k_per_integral}
\psi_k(X;\chi)=\int_{1}^{X} \psi_{k-1}(t;\chi)\, \frac{dt}{t}
\end{equation}
 and the Prime Number Theorem, we have 
$Y- \psi_2(Y) \ll_{D} Y(\log Y)^{-D}$ 
 for all $D \geq 0$. Thus, the bound 
\[
\max_{Y \leq X} \sum_{q \in \DF(Q)}\frac{\big|Y-\psi_2(Y) \big|}{h(q)} \ll Q^{1/2}X (\log X)^{-A'}
\]
follows after splitting the sum into exceptional and non-exceptional discriminants and using the bounds $|\DF_{\text{ex}}(Q)| \ll \log Q$ and $h(q) \gg |q|^{1/2}(\log |q|)^{-1}$ for $q \in~\DF(Q)~\smallsetminus~\DF_{\text{ex}}(Q)$, which we have found earlier.  
As for the term $\big|\psi_2(Y;\chi_q)  \big|$ above, we first note that 
\[
\max_{Y \leq X} \sum_{q \in \DF_{\text{ex}}(Q)}\frac{\big|\psi_2(Y;\chi_q)  \big| }{h(q)} \ll (\log Q) X (\log X)^2
\]
is negligible if $Q$ is not too small, i.e.\ if $Q \geq (\log X)^{2A'+6}$; but if $Q$ is small, then 
\[
\max_{Y \leq X} \sum_{q \in \DF_{\text{ex}}(Q)}\frac{\big|\psi_2(Y;\chi_q)  \big| }{h(q)}
\]
is negligible by the Siegel--Walfisz theorem in the form 
\begin{equation}\label{eq:SW_for_Dirichlet_Char}
\psi_0(X;\chi) 
\ll_{A'} Xe^{-c\sqrt{\log X}},
\end{equation}
which holds with some absolute positive constant $c$ for all $q \leq (\log X)^{2A'+6}$ and all non-principal Dirichlet characters $\chi$ modulo $q$ (see \cite[Corollary 11.18]{MV07}, for example). 
 Thus, it remains to bound the sum over $q \in \DF(Q) \smallsetminus \DF_{\text{ex}}(Q)$ and this may be accomplished by means of the original Bombieri--Vinogradov theorem -- or rather the underlying average character sum that we will also use in Section \ref{sec:realchar1} (compare the bound \eqref{eq:E''_k2} for $E''_{2;k}$ with $\nu=1$ and $k=2$ there). Hence we also get 
\[
\max_{Y \leq X} \sum_{q \in \DF(Q)}\frac{\big|\psi_2(Y;\chi_q)  \big|}{h(q)} \ll Q^{1/2}X (\log X)^{-A'}
\]
if $Q^{2} \leq X(\log X)^{-B'}$ for some $B'=B'(A')$ (which may be chosen as small as $B(A')$ above). In summary, the same bound holds for the second term on the right-hand side of \eqref{eq:Transition_from_psi2_to_psi2} in the same range, which is larger than the range \eqref{eq:range_zum_letzten_Mal_fuer_BV} for which we have bounded the first term. 

 This finishes the proof of \eqref{eq_bound_for_psi2_01} in the range \eqref{eq:range_zum_letzten_Mal_fuer_BV} with $B=16A'+300$ and therefore it also concludes the proof of Theorem \ref{thm:BV_01_unsmoothed}. 


\section{Complex character sums for the Bombieri--Vinogradov type results}\label{sec:complexchar1} 

In this section, we estimate the first term $E'_k(Q,X)$ in \eqref{eq:bv_gesplittet2}. 
Using dyadic decomposition and the class number bound \eqref{eq:classno_bound1} for the discriminants in $M''(Q)$, we get
\begin{equation}\label{eq:bv_gesplittet_complex}
E'_k(Q,X)\ll (\log X)^2 \max_{Y \leq X} \max_{Q_0 \leq Q_1\leq Q}\ Q_1^{-1/2}\sum_{q \in M''(Q_1)} \sum_{\substack{\chi \in \Hhat(q) \\ \chi^2 \neq \chi_0}} \left|\int_{(c)}\frac{L'}{L}(s,\lambda_{\chi})Y^s s^{-(k+1)}\, ds\right|
\end{equation}
for all $c>1$. Like in Section \ref{sec:largesieve}, we set 
\[
\Hhatc(q)=\{\chi \in \Hhat(q) \mid \chi^2 \neq \chiq_0\}
\]
 for all $q \in M''(Q_1)$ and
\[
\Hhatc(Q_1)=\bigcup_{q \in M''(Q_1)} \Hhatc(q).
\] 
Moreover, let $a_\chi(n)$ denote the coefficients of the $L$-series of the logarithmic derivative of $L(s, \lambda_\chi)$, i.e.
\[
\frac{L'}{L}(s,\lambda_{\chi})=\sum_{n \geq 1} \frac{a_\chi(n)}{n^{s}}
\]
 and split it according to Bombieri's modification of Gallagher's identity: 
For every $1 \leq z\leq X$, we set
	\begin{equation*}
	F_z:=F_z(s,\lambda_{\chi}):=\sum_{n \leq z} \frac{a_\chi(n) }{n^{s}},
	\enskip 
	G_z:=G_z(s,\lambda_{\chi}):=\sum_{n > z} \frac{a_\chi(n) }{n^{s}}, 
	\enskip 
	M_z:=M_z(s,\lambda_{\chi}):=\sum_{n \leq z} \frac{b_\chi(n) }{n^{s}},
	\end{equation*}
	where the coefficients $b_\chi(n)$ are the coefficients of ${L(s, \lambda_\chi)^{-1}}$. 
	Then
	\begin{equation}\label{eq:gallagher_bv_complex}
	\frac{L'}{L}=G_z(1-LM_z)+F_z(1-LM_z)+L'M_z.
	\end{equation}
	 Thus, for all $c>1$, we have 
	\begin{equation*}
	\int_{(c)}\frac{L'}{L}(s,\lambda_{\chi})\frac{Y^s}{s^{k+1}}\, ds
	= 
	\int_{(c)}G_z(1-LM_z) \frac{Y^s}{s^{k+1}}\, ds 
	\,+\, \int_{(c)}\Big(F_z(1-LM_z)+L'M_z\Big) \frac{Y^s}{s^{k+1}}\, ds. 
	\end{equation*}
	We may move the line of integration of the second integral into the critical strip because $F_z$ and $M_z$ are Dirichlet polynomials and $L$ and $L'$ are entire functions for all $\chi \in~\Hhatc(Q)$. 
	It will turn out that moving it to 
	\[
	\wc=\wc(\nu)=1-\frac{3}{24-8\nu} \in \bigg[\frac{13}{16}\, ,\, \frac{7}{8}\bigg]
	\]
	maximizes the admissible range for the discriminants in Theorems \ref{thm:BV_01} and \ref{thm:BV_prime}. 
	Repeatedly using the inequality $2|ab|\leq~|a|^2+|b|^2$, we obtain
	\begin{equation}\label{eq:3terms}
	\begin{aligned}
	&\max_{Y \leq X} \sum_{\chi \in \Hhatc(Q_1)}\bigg|\int_{(c)}\frac{L'}{L}(s,\lambda_{\chi})Y^s s^{-(k+1)}\, ds\bigg| \\
	\ll \ &\ X^{c}  \sum_{\chi \in \Hhatc(Q_1)} \int_{(c)}(|G_z|^2+|{1-LM_z}|^2) |s|^{-(k+1)}\, |ds| \\
	+&\ X^{\wc}\sum_{\chi \in \Hhatc(Q_1)} \int_{(\wc)}(1+|F_z|^2+|M_z|^2+|F_z M_z|^2) |s|^{-(k+1)}\, |ds| \\
	 +&\ X^{\wc}\sum_{\chi \in \Hhatc(Q_1)} \int_{(\wc)}(|L|^2+|L'|^2) |s|^{-(k+1)}\, |ds| 
	\end{aligned}
	 \end{equation}
for all $c>1$. The first and second term on the right-hand side will be evaluated by our large sieve inequality for complex class group characters, in particular by Corollary \ref{corr:lsi_02}. Before we can do this, we have to determine the coefficients $a_\chi(n)$ and $b_\chi(n)$ of $F_z$, $G_z$ and $M_z$. 
This is slightly more complicated than in the classical case, since if 
$\chi \in \Hhatc(q)$, then the product
\begin{equation}\label{eq:product-of-coeff}
\lambda_\chi(m) \lambda_\chi(n)= \sum_{d \mid (m,n)} \chi_q(d) \lambda(mnd^{-2}) 
\end{equation}
is not as simple as the product of two values of a Dirichlet character (see \cite[\S 6.6]{Iwa97}, for example; recall that the $\lambda_\chi(n)$ are coefficients of primitive holomorphic cusp forms of weight one, level $q$ and nebentypus $\chi_q$, as we already mentioned in Section \ref{sec:largesieve}). 
This product formula yields the Euler product
\[
L(s, \lambda_\chi)=\prod_{p} \big(1- \lambda_\chi(p)p^{-s}+\chi_q(p)p^{-2s} \big)^{-1}
\]
from which one easily deduces (see \cite[Lemma 2.1]{KM97}) that 
\begin{equation*}\label{eq:formel_L_inverse}
L(s, \lambda_\chi)^{-1}= \sum_{\ell,m \geq 1} \chi_q(m)  \mu(\ell) |\mu(\ell m)| \lambda_\chi(\ell) (lm^2)^{-s}.
\end{equation*}
We thus get the following expressions for the Dirichlet series $M_z$, $F_z$, $G_z$ and $1-LM_z$: 
\begin{align}
M_z(s,\lambda_{\chi}) =&\ \quad \ 
 \sum_\stacksum{\ell,m \geq 1}{\ell m^2 \leq z} \chi_q(m)  \mu(\ell) |\mu(\ell m)| \lambda_\chi(\ell) (lm^2)^{-s}, 
\label{eq:Mz_expansion}\\
F_z(s,\lambda_{\chi}) =&\ 
- \sum_\stacksum{k, \ell, m \geq 1}{k \ell m^2 \leq z} (\log k) \chi_q(m)  \mu(\ell) |\mu(\ell m)| \lambda_\chi(\ell) \lambda_\chi(k) (klm^2)^{-s}, 
\label{eq:Fz_expansion}\\
G_z(s,\lambda_{\chi}) =&\ 
- \sum_\stacksum{k, \ell, m \geq 1}{k \ell m^2 > z} (\log k) \chi_q(m)  \mu(\ell) |\mu(\ell m)| \lambda_\chi(\ell) \lambda_\chi(k) (klm^2)^{-s}, 
\label{eq:Gz_expansion}\\
1-LM_z(s,\lambda_{\chi}) =&\  
- \sum_{\substack{k, \ell, m \geq 1 \\ \ell m^2 \leq z \\ \, k\ell m^2 > z}}
\chi_q(m)  \mu(\ell) |\mu(\ell m)| \lambda_\chi(\ell) \lambda_\chi(k) (klm^2)^{-s}. \label{eq:1-LMz_expansion}
\end{align}
These series are not yet in the right form for a direct application of Corollary \ref{corr:lsi_02}, but the following (in)equalities 
will bring them into the right shape:

\begin{lemma}\label{lem:kowmich97_lemmas}
For all positive integers $\ell$ and $m$, let $A(\ell,m)$ be a complex number. 
\begin{enumerate}
\item
Let $\alpha >0$. Assume that $\big|\sum_{\ell \geq 1} A(\ell,m) \ell^{-(1+\alpha)+it} \big| \ll m^2$ for 
all $t \in \Rr$.  
Then
\begin{equation}\label{eq:kowmich97_lemmas_1}
\bigg| \sum_{\ell,m \geq 1} A(\ell,m) (\ell m^2)^{-(1+\alpha)+it} \bigg|^2 \ll 
\alpha^{-1} \sum_{m \geq 1} m^{-3-2\alpha} \bigg| \sum_{\ell \geq 1} A(\ell,m) \ell^{-(1+\alpha)+it} \bigg|^2.
\end{equation}
\item
Let $\wc \in \big[\frac{13}{16}\, ,\, \frac{7}{8}\big]$. Assume that $\big|\sum_{\ell \geq 1} A(\ell,m) \ell^{-\wc+it} \big| < \infty$ for all $m \geq 1$ and all $t \in \Rr$. Moreover, assume that there exists a real number $M$ such that $A(\ell,m)=0$ for all $m \geq M$ and all $\ell \geq 1$. Then
\begin{equation}\label{eq:kowmich97_lemmas_2}
\bigg| \sum_{\ell,m \geq 1} A(\ell,m) (\ell m^2)^{-\wc+it} \bigg|^2 \ll \sum_{m \leq M} m^{-2\wc} \bigg| \sum_{\ell \geq 1} A(\ell,m) \ell^{-\wc+it} \bigg|^2.
\end{equation}
\item
Let $\chi \in \Hhatc(q)$ and $j_1,j_2,j_3 \geq 1$. Then
\begin{equation}\label{eq:kowmich97_lemmas_3}
\sum_\stacksum{\ell,m \geq 1}{\ell m \geq j_1} A(\ell,m) \lambda_\chi(\ell) \lambda_\chi(m) (\ell m)^{-s}= \sum_\stacksum{h,d \geq 1}{hd^2 \geq j_1} \chi_q(d) \sum_\stacksum{v,w \geq 1}{vw=h} A(vd,wd) \lambda_\chi(h) (hd^2)^{-s}
\end{equation}
and
\begin{equation}\label{eq:kowmich97_lemmas_3b}
\sum_{\ell \leq j_2}\sum_{m \leq j_3} A(\ell,m) \lambda_\chi(\ell) \lambda_\chi(m) (\ell m)^{-s}= \sum_{h,d \geq 1} \chi_q(d) \sum_{v \leq \frac{j_2}{d}}\sum_\stacksum{w \leq \frac{j_3}{d}}{vw=h} A(vd,wd) \lambda_\chi(h) (hd^2)^{-s}
\end{equation}
for all $s \in \Cc$ for which the series converge.
\end{enumerate}
\end{lemma}

\begin{proof}
By the Cauchy--Schwarz inequality, we have
\begin{equation}\label{eq:kowmich97_lemmas_Cauchy_S}
\bigg| \sum_{\ell,m \geq 1} A(\ell,m) (\ell m^2)^{-s} \bigg|^2 \leq \bigg(\sum_{m \geq 1} m^{2(r-\Reel(s))} \bigg) \bigg(\sum_{m \geq 1} m^{-2(r+\Reel(s))} \bigg| \sum_{\ell \geq 1} A(\ell,m) \ell^{-s} \bigg|^2 \bigg)
\end{equation}
for all real numbers $r$ and all complex numbers $s$ for which the sums on the right side converge.
 The first  bound follows for $r=\frac 1 2$ and $s=(1+\alpha)-it$. 
As for the second bound, the sums on the right side of \eqref{eq:kowmich97_lemmas_Cauchy_S} are then only over $m \leq M$; 
the bound follows for $r=0$ and $s=\wc-it$. 
The equalities in (c)  follow from \eqref{eq:product-of-coeff}.
\end{proof}

\begin{remf}
These (in)equalities have been used in \cite[\S 7]{KM97} to prove a zero-density estimate for $L$-functions associated to certain cusp forms. The first proofs of the Bombieri--Vinogradov theorem relied heavily on zero-density estimates for Dirichlet $L$-functions; Gallagher's simplification of these proofs then removed any direct appeal to the zeros but still kept the core of the argument. Thus, it is not surprising that Lemma \ref{lem:kowmich97_lemmas} plays a role both here and in~\cite{KM97}.
\end{remf}

Set $\alpha=(\log X)^{-1}$ and $c=1+\alpha$, then apply \eqref{eq:kowmich97_lemmas_1} and \eqref{eq:kowmich97_lemmas_3} to \eqref{eq:Gz_expansion} and obtain
\begin{equation}\label{eq:Gz_hoch2}
\begin{aligned}
& |G_z(c+it,\lambda_{\chi})|^2 \ll \ 
(\log X) \sum_{m \geq 1} m^{-3-2 \alpha} \bigg|\sum_\stacksum{k,\ell \geq 1}{k \ell  > \frac{z}{m^2}} (\log k)  \mu(\ell) |\mu(\ell m)| \lambda_\chi(\ell)  \lambda_\chi(k) (kl)^{-(c+it)} \bigg|^2 \\
=\ &
(\log X) \sum_{m \geq 1} m^{-3-2 \alpha} \bigg|\sum_\stacksum{h, d \geq 1}{hd^2  > \frac{z}{m^2}} \chi_q(d) \sum_\stacksum{v,w \geq 1}{vw=h} (\log vd) \mu(wd) |\mu(wdm)| \lambda_\chi(h) (hd^2)^{-(c+it)} \bigg|^2.
\end{aligned}
\end{equation}
Set 
\begin{equation}\label{eq:a_1(h,d,m)}
a_1(h,d,m)= \sum_\stacksum{v,w \geq 1}{vw=h} (\log vd) \mu(wd) |\mu(wdm)| 
\end{equation}
and apply once again \eqref{eq:kowmich97_lemmas_1} to the right side of \eqref{eq:Gz_hoch2}. This yields
\[
|G_z(c+it,\lambda_{\chi})|^2 \ll (\log X)^2 \sum_{m \geq 1} m^{-3-2 \alpha} \sum_{d \geq 1} d^{-3-2 \alpha} \bigg|\sum_{h > \frac{z}{m^2d^2}} a_1(h,d,m) \lambda_\chi(h) h^{-(c+it)} \bigg|^2,
\]
which now has the right form to apply \eqref{eq:LSI_int_uncond}. We get
\begin{equation*}
\begin{aligned}
&\sum_{\chi \in \Hhatc(Q_1)} \int_{(c)}|G_z(s,\lambda_{\chi})|^2 |s|^{-(k+1)}\, |ds| \\
\ll_\eps\ & Q_1^\eps (\log X)^2 \sum_{m,d \geq 1} (md)^{-3-2 \alpha} \sum_{h > \frac{z}{m^2d^2}} |a_1(h,d,m)|^2 (h^{-1-2\alpha}+h^{-3/2-2\alpha}Q_1^{5/2}) (1+(\log h)^3).
\end{aligned}
\end{equation*}
Since $z^\alpha \leq X^\alpha \ll 1$ and
\[
|a_1(h,d,m)|^2 \leq \tau(h)^2 (\log hd)^2 
\]
for all $h$, $d$ and $m$, the contribution coming from $|G_z|^2$ in \eqref{eq:3terms} is bounded by
\begin{equation}\label{eq:contribution_from_Gz}
O_\eps \big(X (\log X)^{K_1} Q_1^\eps (1+Q_1^{5/2}z^{-1/2}))
\end{equation}
for some $K_1>0$; 
in fact, we may choose $K_1=11$. 

A comparison of \eqref{eq:Gz_expansion} and \eqref{eq:1-LMz_expansion} shows that the analysis of the contribution coming from $|1-LM_z|^2$ in \eqref{eq:3terms} can be performed in almost exactly the same way and the same bound is obtained. 
 Thus we record that the whole first term on the right side of \eqref{eq:3terms} can be bounded by \eqref{eq:contribution_from_Gz}. 

\medskip
Moving on to the second line of \eqref{eq:3terms}, each summand in the integrand is again analysed separately. The contribution coming from the integrand $1$ follows directly from \eqref{eq:MengeAllerCharaktereBound}: 
\begin{equation}\label{eq:contribution_from_1}
X^{\wc}  \sum_{\chi \in \Hhatc(Q_1)} \int_{(\wc)}1 \cdot |s|^{-(k+1)}\, |ds| \ll X^{\wc} |\Hhatc(Q_1)| \ll X^{\wc} Q_1^{3/2} (\log Q_1).
\end{equation}
Next, $F_z$ and $M_z$ are bounded in the same way as $G_z$ but with appeal to \eqref{eq:kowmich97_lemmas_2} (with $M=z$) instead of \eqref{eq:kowmich97_lemmas_1}, \eqref{eq:kowmich97_lemmas_3b} instead of \eqref{eq:kowmich97_lemmas_3} and \eqref{eq:LSI_int_uncond_with_trivial} instead of  \eqref{eq:LSI_int_uncond}. In fact, with $a_1(h,d,m)$ given by \eqref{eq:a_1(h,d,m)}, we find
\begin{equation}\label{eq:contribution_from_Fz}
\begin{aligned}
&X^{\wc}  \sum_{\chi \in \Hhatc(Q_1)} \int_{(\wc)}|F_z(s,\lambda_{\chi})|^2 |s|^{-(k+1)}\, |ds| 
\\
\ll_\eps\ & 
X^{\wc} Q_1^\eps \sum_{m,d \leq z} (md)^{-2\wc} 
\Big(
Q_1^{3/2}\sum_{h \leq Q_1^2} |a_1(h,d,m)|^2 h^{1-2\wc} (1+(\log h)^3) 
\\
& \qquad \qquad \qquad \qquad \quad \enskip + \sum_{Q_1^2<h \leq \frac{z}{m^2d^2}} |a_1(h,d,m)|^2 (h^{1-2\wc}+h^{1/2-2\wc}Q_1^{5/2}) (\log h)^3 
\Big)
\\
\ll\ & X^{\wc} (\log X)^{K_2} Q_1^\eps (Q_1^{11/2-4\wc}+z^{2-2\wc})
\end{aligned}
\end{equation}
for some $K_2 \geq 0$; we may choose $K_2=8$. Similarly,
\begin{equation}\label{eq:contribution_from_Mz}
X^{\wc}  \sum_{\chi \in \Hhatc(Q_1)} \int_{(\wc)}|M_z(s,\lambda_{\chi})|^2 |s|^{-(k+1)}\, |ds| 
\ll
X^{\wc} (\log X)^{K_2} Q_1^\eps (Q_1^{11/2-4\wc}+z^{2-2\wc}).
\end{equation}
The integrand $|F_zM_z|^2$ requires a little bit more work, but the approach is familiar by now: By~\eqref{eq:Mz_expansion}, \eqref{eq:Fz_expansion}, \eqref{eq:kowmich97_lemmas_2} and \eqref{eq:kowmich97_lemmas_3b}, we have
\begin{equation*}
\begin{aligned}
&\big|F_z\big(\wc+it,\lambda_{\chi}\big)M_z\big(\wc+it,\lambda_{\chi}\big)\big|^2 \\
\ll\ & \sum_{m,w,d \leq z} (mwd)^{-2\wc} \Big| \sum_{v \leq \frac{z}{w^2}} \sum_{b \leq \frac{z}{(md)^2}} a_2(m,b,v,w) \lambda_\chi(b) \lambda_\chi(v) (bv)^{-(\wc+it)}\Big|^2,
\end{aligned}
\end{equation*}
where
\[
a_2(b,d,m,v,w)=\mu(v) |\mu(vw)| a_1(b,d,m). 
\]
By \eqref{eq:kowmich97_lemmas_3b} and \eqref{eq:kowmich97_lemmas_2}, we then get
\begin{equation*}
\begin{aligned}
&\big|F_z\big(\wc+it,\lambda_{\chi}\big)M_z\big(\wc+it,\lambda_{\chi}\big)\big|^2 \\
\ll\ &  \sum_{m,w,d,r \leq z} (mwdr)^{-2\wc} 
\Big| \sum_{h \geq 1} a_3(h,r,d,m,w) \lambda_\chi(h)  h^{-(\wc+it)}\Big|^2,
\end{aligned} 
\end{equation*}
where
\[
a_3(h,r,d,m,w)=
\sum_{v' \leq \frac{z}{w^2r}} \sum_\stacksum{b' \leq \frac{z}{(md)^2r}}{v'b'=h}
a_2(b'r,d,m,v'r,w)
\]
whose absolute value is 
\[
|a_3(h,r,d,m,w)| \leq	\sum_{v' \leq \frac{z}{w^2r}} \sum_\stacksum{b' \leq \frac{z}{(md)^2r}}{v'b'=h} \tau(b'r) (\log b'rd) \ll (\log z)\, \tau(r)\, \tau_3(h),
\]
where $\tau_3(h)$ is the ternary divisor function (i.e., the number of ordered $3$-tuples $(b_1,b_2,b_2)$ of positive integers such that $h=b_1b_2b_3$). By Corollary \ref{corr:lsi_02} and the bound 
\[
\sum_{h \leq z^2} \tau_3(h)^2 \ll z^2(\log z)^{8},
\] 
which follows by the method of \cite[p.~37]{Kow04}, for example, we obtain
\begin{equation}\label{eq:contribution_from_FzMz}
X^{\wc}  \sum_{\chi \in \Hhatc(Q_1)} \int_{(\wc)}|F_z\big(s,\lambda_{\chi}\big)M_z(s,\lambda_{\chi})|^2 |s|^{-(k+1)}\, |ds| 
\ll_\eps\ X^{\wc} (\log X)^{K_3} Q_1^\eps (Q_1^{11/2-4\wc}+z^{4-4\wc})
\end{equation}
for some $K_3 \geq 0$; we may choose $K_3=13$. 

\medskip
We gather the bounds \eqref{eq:contribution_from_1}, \eqref{eq:contribution_from_Fz}, \eqref{eq:contribution_from_Mz}, \eqref{eq:contribution_from_FzMz} and record that the contribution to the right side of \eqref{eq:3terms} coming from the second line is
\begin{equation}\label{eq:contribution_from_Fz,Mz,1,FzMz}
O_\eps \big(X^{\wc} (\log X)^{K_3} Q_1^\eps (Q_1^{11/2-4\wc}+z^{4-4\wc}) \big).
\end{equation}

\medskip
It remains to bound the third term on the right side of \eqref{eq:3terms}. We could proceed as in \cite{Bom87}, using the bound $\sum_{n \leq N} \lambda_\chi(n) \ll_{\eps} (|q|^2N)^{1/2+ \eps}$ that holds for Fourier coefficients of weight-one cusp forms and therefore for our coefficients $\lambda_\chi$ as they arise from complex class group characters here (see Proposition 5 in \cite{HM06}, for example). 

However, in our case it is sufficient and easier to use the convexity bound for the functions $L(s,\lambda_\chi)$: Each of them satisfies a functional equation of the form 
\begin{equation*}\label{eq:functiona-equation-for-classgroupLfcts}
\Phi(s,\lambda_\chi) =\Phi(1-s, \lambda_\chi),
\end{equation*}
 where
\[
\Phi(s,\lambda_\chi)=\bigg(\frac{\sqrt{|q|}}{2\pi}\bigg)^{s}\Gamma(s) L(s, \lambda_\chi);
\]
see \cite[\S22.3]{IK04}, for example. 
Therefore, the convexity principle of Phragm\'en--Lindel\"of 
 yields 
\begin{equation*}
L(\wc +it, \lambda_\chi) \ll_\eps \big(|q|^{1/2} (1+|t|)\big)^{1-\wc+\eps},
\end{equation*}
for all $t \in \Rr$. 

Combining the convexity principle for $L(s, \lambda_\chi)$ and Cauchy's inequality for the derivative of analytic functions (consider the disc around $\wc+it$ with radius $(\log Q_1)^{-1}$), 
we also get
\[
L'(s, \lambda_\chi) \ll_\eps |q|^{(1-\wc)/2+\eps} (1+|t|)^{1-\wc+\eps+(\log Q_1)^{-1}}(\log Q_1).
\]
If $k\geq 2$, these bounds and \eqref{eq:MengeAllerCharaktereBound} yield
\begin{equation}\label{eq:ComplexChar3Term}
\sum_{\chi \in \Hhatc(Q_1)} \int_{(\wc)}(|L(s,\lambda_{\chi})|^2+|L'(s,\lambda_{\chi})|^2) |s|^{-(k+1)}\, |ds| \ll_\eps (\log Q_1)^3 Q_1^{5/2-\wc+\eps}. 
\end{equation}
\begin{remf}
 Duke, Friedlander and Iwaniec \cite[Theorem 2.6]{DFI02} proved the first subconvexity bound for the $L$-functions associated to complex class group characters for all fundamental discriminants (they had previously proved such a bound for special types of discriminants). Subsequently, a simpler proof -- and a slightly better bound -- was found by Blomer, Harcos and Michel \cite[Corollary~1]{BHM07}. As is clear from the theorem numbering of these results, these are only special cases of subconvexity bounds for much more general $L$-functions.
  The convexity bound is more than enough for our needs and any invocation of these deep results would be pretentious here.
\end{remf}

\medskip 
Let $K=A+2+K_4$ for some $K_4 \geq \max(K_1,K_2,K_3,3)$; thus, $K=A+15$ is admissible, for example. 
 We put together the upper bounds  
\eqref{eq:contribution_from_Gz}, \eqref{eq:contribution_from_Fz,Mz,1,FzMz} 
and \eqref{eq:ComplexChar3Term} that we have found for the three summands in \eqref{eq:3terms}, insert them into \eqref{eq:bv_gesplittet_complex}
and get
\begin{equation}\label{eq:E'_k_nacheinsetzen}
\begin{aligned}
E'_k(Q,X)\ll_{\eps} (\log X)^{K-A} &\max_{Q_0 \leq Q_1\leq Q} Q_1^{-1/2+\eps}X^{\wc} \\
 &\quad\quad\quad \times \left( X^{1-\wc}(1+Q_1^{5/2} z^{-1/2}) +Q_1^{11/2-4\wc}+z^{4-4\wc}+ Q_1^{5/2-\wc}\right).
\end{aligned}
\end{equation}
Set 
\begin{equation}\label{eq:setze_log_power}
D:=\max(D_1, 4K)
\end{equation} 
in Remark \ref{rem:bv_small_discr}. We can assume without loss of generality that $\eps \leq \frac 1 4$. 
If $Q \geq Q_0=(\log X)^D$, then $Q_1^{1/2-\eps}$ 
is therefore at least $(\log X)^K$ and if we choose $z=Q_1^{4-\nu+2\eps}(\log X)^{2K}$
and $\wc=~\wc(\nu)=1-\frac{3}{24-8\nu}$, 
we get 
\begin{align*}
E'_k(Q,X)  
\ll_{\eps} \ 
(\log X)^{-A} Q^{\nu/2}(X+(\log X)^{8K(1-\wc)+K}X^{\wc}Q^{5-\nu/2-4\wc+9\eps}
). 
\end{align*}
This gives
\[
E'_k(Q,X) \ll_{\eps} Q^{\nu/2}X (\log X)^{-A}
\]
if 
$Q^{4(1+(2-\nu)(3-\nu)/3)+72\eps} \leq X (\log X)^{-B}$ with $B = 8K+\max_{\nu \in [0,1]} \frac{K}{1-\wc(\nu)}=16K$; that is, we may choose 
\begin{equation}\label{eq:bv_B-Wert_complex}
B=16A+300.
\end{equation}

\begin{remarkf}\label{rem:BV-Linde01}
If we assume the Lindel\"of Hypothesis, we may use the conditional large sieve inequality \eqref{eq:LSI_int_cond} instead of \eqref{eq:LSI_int_uncond_with_trivial} and \eqref{eq:LSI_int_uncond} and replace the exponent $\frac{5}{2}-\wc+\eps$ by $\frac{3}{2}+\eps$ in \eqref{eq:ComplexChar3Term}. This leads to the bound
\begin{equation*}\label{eq:E'_k_nacheinsetzen_Linde}
\begin{aligned}
E'_k(Q,X)\ll_{\eps} (\log X)^{K-A} \max_{Q_0 \leq Q_1\leq Q}\ & Q_1^{-1/2+\eps}X^{\wc} \\
 & \times \left( X^{1-\wc}(1+Q_1^{3/2} z^{-1/2}) +Q_1^{3/2}+z^{4-4\wc}\right)
\end{aligned}
\end{equation*}
for some $\wc = \wc(\nu)\geq \frac{3}{4}$, and this yields
\[
E'_k(Q,X) \ll_\eps Q^{\nu/2}X (\log X)^{-A}
\]
if 
$\nu \geq \demi$  and $Q^{4-2\nu+\eps} \leq X (\log X)^{-B}$, or if $\nu  < \demi$ and $Q^{4(2-\nu)^2/3+\eps} \leq X (\log X)^{-B}$. Since these ranges are shorter than the unconditional one in the next section, they yield the ranges in Remark \ref{rem:BV_Linde2}  and the first statement in Remark \ref{rem:conditional_ranges} (with $\nu=1$).
\end{remarkf}


\section{Real character sums for the Bombieri--Vinogradov type results}\label{sec:realchar1}

Before approaching the second sum $E''_k(Q,X)$ in \eqref{eq:bv_gesplittet2}, we note that each of the Chebyshev functions $\wpsi(X;q,\chi)$ for real class group characters $\chi$ can be written as the sum of two Chebyshev functions for Dirichlet characters: 
If $q \in \DF$ and $\chi \in \Hhat(q)$ is a real class group character, 
then the Kronecker Factorization Formula (see \cite[Theorem 12.7]{Iwa97}, for example) states that there exist 
two (positive or negative) fundamental discriminants  $d_1$ and $d_2$ with $d_1d_2=q$ such that the  
$L$-function of $\chi$ factors as $L(s,\lambda_{\chi})=L(s,\chi_{d_1}) L(s,\chi_{d_2})$ 
into the Dirichlet $L$-functions that are associated to the primitive real characters $\chi_{d_1}$ modulo $|d_1|$ and $\chi_{d_2}$ modulo $|d_2|$. On the other hand, every such factorization of $q$ gives rise to a real class group character $\chi=\chi_{d_1,d_2}$ of $\Hcl(q)$. Note that the trivial class group character $\chiq_0$ corresponds to the trivial factorization $q=1\cdot q$. Thus 
\begin{equation}\label{eq:Kronecker2}
\frac{L'(s,\lambda_{\chi_{d_1,d_2}})}{L(s,\lambda_{\chi_{d_1,d_2}})}=\frac{L'(s,\chi_{d_1})}{L(s,\chi_{d_1})}+\frac{L'(s,\chi_{d_2})}{L(s,\chi_{d_2})}\ .
\end{equation}
Let $F_{}(Q)$ denote the set of all (positive or negative) fundamental discriminants $d\neq 1$ with $|d|\leq Q$. The $k$-th iteration of the Mellin transform of ${\psi}_k(X;\chi_{d})$, which was defined in \eqref{eq:psik_for_Dirichlet}, is $\frac{L'}{L}(s,\chi_{d})$, hence  
\begin{equation}\label{eq:mellin_dirich}
{\psi}_k(X;\chi_{d})=-\frac{1}{2\pi i}\int_{(c)}\frac{L'}{L}(s,\chi_{d})X^s s^{-(k+1)}\, ds
\postdisplaypenalty=100
\end{equation}
for each $c>1$. 
Therefore, \eqref{eq:Kronecker2} and \eqref{eq:mellin_dirich} imply
\begin{align*}
E''_k(Q,X) \leq&\ \sum_{d_1 \in F_{}(Q)} \sum_{\substack{d_2 \in F_{}(Q)\\ d_1 d_2 \in M''(Q)}} \frac{1}{h(d_1d_2)}(|{\psi}_k(X;\chi_{d_1})|+ |{\psi}_k(X;\chi_{d_2})|)\\
=&\ 2 \sum_{d_1 \in F_{}(Q)} |{\psi}_k(X;\chi_{d_1})| \sum_{\substack{d_2 \in F_{}(Q)\\ d_1 d_2 \in M''(Q)}} \frac{1}{h(d_1d_2)}. 
\end{align*}
The class number bound \eqref{eq:classno_bound1} for the discriminants in $M''(Q)$ yields 
\[
E''_k(Q,X) \ll \  (\log Q)  \sum_{d_1 \in F_{}(Q)} \frac{1}{|d_1|^{1/2}} |{\psi}_k(X;\chi_{d_1})| \sum_{\substack{d_2 \in F_{}(Q)\\ d_1 d_2 \in M''(Q)}} \frac{1}{|d_2|^{1/2}}.
\]
By the assumption \eqref{Bedingung_M} for $M(Q)$ in Theorem \ref{thm:BV_01} and for $\Pi(Q)$ in Theorem \ref{thm:BV_prime}, the sum over $d_2$ has at most $Q^\nu$ terms. Hence 
\[
E''_k(Q,X) \ll \ (\log Q)    \sum_{d_1 \in F_{}(Q)} \frac{1}{|d_1|^{1/2}} |{\psi}_k(X;\chi_{d_1})| \sum_{\substack{d_2 \leq \min(Q^\nu,\frac{Q}{|d_1|})}} \frac{1}{d_2^{1/2}}, 
\]
which implies, by dyadic decomposition, 
\begin{equation*}\label{eq:E_k'_2}
\begin{aligned}
E''_k(Q,X) 
\ll \quad  & \quad \enskip (\log Q)^2    Q^{\nu/2}  \max_{Q_1\leq Q^{1-\nu}}\ Q_1^{-1/2} \sum_{d_1 \in F_{}(Q_1)} |{\psi}_k(X;\chi_{d_1})| 
\\
 & +\  (\log Q)^2  Q^{1/2}  \max_{Q^{1-\nu} \leq Q_1\leq Q}\ Q_1^{-1} \sum_{d_1 \in F_{}(Q_1)} |{\psi}_k(X;\chi_{d_1})| 
\\
=\quad & \enskip  E''_{1;k}(Q,X) + E''_{2;k}(Q,X), \ \text{ say}.
\end{aligned}
\end{equation*}
Note that we cannot profit here from the fact that $M''(Q)$ does not contain any small discriminants, which were already handled by means of Goldstein's generalization of the Siegel--Walfisz theorem. Instead, we may use the original Siegel--Walfisz theorem to handle the small discriminant divisors $d_1$ here.

\medskip
In fact, we have now basically reduced the problem to the analogous problem for Dirichlet characters, i.e.\ 
we are in a similar position as in the original Bombieri--Vinogradov theorem, the only differences being:
\begin{enumerate}
\item
The first term $E''_{1;k}(Q,X)$ above has the factor $Q_1^{-1/2}$ in front of the sum 
(coming from the class number estimate) instead of $Q_1^{-1}$ (coming from the Euler totient function estimate) in the classical case. 
This will lead to a smaller admissible $Q$ for $\nu <1$. 
 \item
Our sums are only over real primitive characters modulo $|d_1|$ with $|d_1| \leq Q$; by positivity, we can, of course, include the non-real primitive Dirichlet characters as well. 
\end{enumerate}

We proceed 
 like in Section \ref{sec:complexchar1}, but using the large sieve inequality for Dirichlet characters. We skip the explicit calculations as they are the same as in \cite{Bom87} and obtain (compare the inequality at the bottom of page 62 and the top of page 63 in \cite{Bom87}): 
\begin{align*}
\sum_{d_1 \in F_{}(Q_1)} |{\psi}_k(X;\chi_{d_1})| \ll \ & X(\log X)^4+X(\log X)^4 Q_1^{2}z^{-1}+X^{1/2}(\log X)^6 z^2 \\
& +X^{1/2}(\log X)^6 Q_1^{2}+X^{1/2}(\log X)^2 Q_1^{4}z^{-2}=:G(X,Q_1,z).
\end{align*}
Here the variable $z$ is the ordinate at which we truncate the inverse Mellin transform in the corresponding Gallagher identity (compare \eqref{eq:gallagher_bv_complex}) and it will be chosen in a moment. 

We obtain
\begin{align*}
E''_{1;k}(Q,X) \ll Q^{\nu/2} (\log X)^2 \max_{Q_1\leq Q^{1-\nu}}\  Q_1^{-1/2} G(X,Q_1,z) 
\end{align*}
and we want to bound the right-hand side with $Q^{\nu/2}X(\log X)^{-A}$. This can be achieved when
we set 
$ z=Q_1^{3/2}(\log X)^{6+A}$ 
if the maximum above is attained for
\[
(\log X)^{12+2A} \leq Q_1 \leq X^{1/5} (\log X)^{-8-6A/5}. 
\]
If the maximum is attained for a smaller $Q_1$, we use the relation \eqref{eq:von_k+1_zu_k_per_integral} and the Siegel--Walfisz theorem \eqref{eq:SW_for_Dirichlet_Char} to get the desired bound.  Altogether, we thus have
\begin{equation}\label{eq:E''_k1}
E''_{1;k}(Q,X) \ll_A Q^{\nu/2}X (\log X)^{-A}
\end{equation}
if 
$Q^{5-5\nu} \leq X(\log X)^{-B}$ for some $B=B(A)>0$.

Similarly,
\begin{align*}
E''_{2;k}(Q,X) \ll Q^{1/2} (\log X)^2 \max_{Q^{1-\nu} \leq Q_1\leq Q}\  Q_1^{-1} G(X,Q_1,z)
\end{align*}
is bounded by $Q^{\nu/2}X (\log X)^{-A}$ if
we set
$z=Q_1Q^{1/2-\nu/2}(\log X)^{6+A}$ and if the maximum is attained for 
\[
(\log X)^{12+2A} \leq Q_1  \quad \text{ and } \quad Q\leq X^{1/(5-3\nu)} (\log X)^{-(40-6A)/(5-3\nu)}. 
\]
Together with the Siegel--Walfisz theorem this leads to the bound  
\begin{equation}\label{eq:E''_k2}
E''_{2;k}(Q,X) \ll_A Q^{\nu/2}X (\log X)^{-A}
\end{equation}
if 
$Q^{5-3\nu} \leq X(\log X)^{-B}$ for some $B=B(A)>0$. Since this range is shorter than the range for $E''_{1;k}(Q,X)$ in \eqref{eq:E''_k1}, we have
\begin{equation*}\label{eq:E'_k1}
E''_{k}(Q,X) \ll_A Q^{\nu/2}X (\log X)^{-A}
\end{equation*}
if 
$Q^{5-3\nu} \leq X(\log X)^{-B}$ for some $B=B(A)>0$. Note that we may choose $B=40+6A$, which is smaller than the $B$-value~\eqref{eq:bv_B-Wert_complex} that we have found at the end of Section \ref{sec:complexchar1}. 

\begin{remarkf}\label{rem:heathb}
We could also employ Heath-Brown's large sieve inequality for real Dirichlet characters \cite{Hea95}
when $\nu<1$. 
This inequality yields then a larger range for the discriminants in this section, but it requires a more careful analysis due to the distinct form of the sum on the right side of the inequality.  
Since we are anyway limited by the much shorter range coming from $E'_k$, 
 this gives no overall gain and therefore we will not delve into this. Note that 
this large sieve inequality 
does not seem to be applicable for $\nu=1$: It yields a term of size $X^{1+\eps}$ (for any $\eps>0$) that  
does not permit us to beat trivial bounds (compare Remark \ref{rem:BV_comparison_to_trivial})  
since our method can only compensate powers of $(\log X)$ when $\nu=1$, but not a genuine~$X^\eps$. 
\end{remarkf}


\section{A general result of Barban--Davenport--Halberstam type}\label{sec:bdh}

In the Barban--Davenport--Halberstam theorem for arithmetic progressions \eqref{eq:BDH_orig}, the prime counting function can be replaced by many other arithmetic functions $g$. Indeed, it suffices to show that $g$ is well distributed in arithmetic progressions to small moduli in order to prove that 
$g$ shows a similar behaviour for almost all residue classes to almost all large moduli (see \cite[\S 17.4]{IK04}, for example). 

We will show here that a general mean square distribution result also holds with respect to binary quadratic forms for arithmetic functions $g$ that are weighted with the function $w(C,n)$ (see \eqref{eq:weight_w-Cn_for_primepowers}), satisfy Siegel--Walfisz conditions for both arithmetic progressions and form classes, and for which the sums 
\[
\sum_{n \leq X} g(n) \sum_\stacksum{1 <k,m<n}{km=n} \chi_1(k) \chi_2(m)
\]
are small for most pairs $(\chi_1,\chi_2)$ of distinct primitive real Dirichlet characters: 
\begin{theorem}\label{thm:bdh_general}  
Let $3 \leq Q \leq X$, let $M(Q)$ be any subset of $\DF(Q)$ and let $g$ be an arithmetic function. Assume that 
\begin{equation}\label{eq:bdh_SW_assumption_forms}
\begin{aligned}
D(g;X;q,C):=&\ \sum_\stacksum{n\leq X}{n \in \Repr(q,C)} w(C,n) g(n) 
- \frac{1}{h(q)}\sum_{K \in \Kcl(q)}\sum_\stacksum{n\leq X}{n \in \Repr(q,K)} w(K,n)g(n) \\
\ll_{L} &\  X^{1/2} (\log X)^{-L}  \Bigg(\sum_{n \leq X} |g(n)|^2 \Bigg)^{1/2}
\end{aligned}
\end{equation}
for all $L>0$, all $q \in \DF(Q)$ with $|q| \leq (\log X)^{L}$ and all form classes $C \in \Kcl(q)$.  

\noindent Also assume that
\begin{equation}\label{eq:bdh_SW_assumption_residues}
\sum_\stacksum{n\leq X}{n \equiv a \mods q} g(n) - \frac{1}{\varphi(q)}\sum_\stacksum{n\leq X}{(n,q)=1} g(n) \ll_L X^{1/2} (\log X)^{-L}  \Bigg(\sum_\stacksum{n \leq X}{(n,q)=1} |g(n)|^2 \Bigg)^{1/2}
\end{equation}
for all $L>0$, all  $q \in \DF(Q)$ and all integers $a$ with $(a,q)=1$. Set
\begin{equation*}
R(g,Q,X):= \sum_{|d_1|>1} \sum_\stacksum{|d_2|>1}{d_1d_2 \in M(Q)} \frac{1}{h(d_1d_2)} \bigg|\sum_{n \leq X} g(n) \sum_\stacksum{1 <k,m<n}{km=n} \chi_{d_1}(k) \chi_{d_2}(m)\bigg|^2,
\end{equation*}
where the outer sums run over (positive and negative) fundamental discriminants and $\chi_{d}$ denotes the primitive real Dirichlet character modulo $|d|$. 

\noindent 
Then 
	\begin{equation}\label{eq:bdh_general}
	\begin{aligned}
	& \sum_{q \in M(Q)} \sum_{C \in \Kcl(q)} \left|D(g;X;q,C)\right|^2
	\\
	\ll_{A,\eps} & 
	\ Q^{1/2}X^{1/2}\left(Q^{3/2+\eps}(\log X)^2 + X^{1/2}(\log X)^{-A}\right) \sum_{n \leq X} |g(n)|^2 +R(g,Q,X)
	\end{aligned}
	\end{equation}
	for all arbitrarily large $A>0$ and all arbitrarily small $\eps>0$.  
\end{theorem}
\noindent We will prove this result in the next section. 

\medskip
Theorem \ref{thm:bdh1} follows easily from Theorem \ref{thm:bdh_general} by appeal to the Siegel--Walfisz theorem for arithmetic progressions and Blomer's variant of it for binary quadratic forms:
\begin{theorem}[Siegel--Walfisz, {\cite[Corollary 11.21]{MV07}}]\label{thm:SW_original}
For any $A>0$, there exists a number $c=c(A) > 0$ 
such that
\[
\pi(X;q,a)=\frac{\li(X)}{\varphi(q)} + O \left(X e^{-c\sqrt{\log X}}\right),
\]
uniformly for all pairs of positive integers $a$ and $q$ with $(a, q) = 1$ and $q \leq (\log X)^A$.
\end{theorem}
\begin{theorem}[Blomer, {\cite[Lemma 3.1]{Blo04}}]\label{thm:blomer}
For any $A>0$, there exists a number\linebreak $c=c(A)>0$ such that
\[
\pi(X;q,C) = \frac{\li(X)}{e(C)h(q)} +O\Big(X e^{-c\sqrt{\log X}} \Big)
\]
uniformly for all $q \in \DF$ with $|q| \leq (\log X)^{A}$ and all $C \in \Kcl(q)$.
\end{theorem}

So let $g$ be the characteristic function of the primes.  Assumption \eqref{eq:bdh_SW_assumption_residues} holds by Theorem~\ref{thm:SW_original} and the Prime Number Theorem.
 As for assumption \eqref{eq:bdh_SW_assumption_forms}, we have
\begin{equation*}\label{eq:bdh_SW_assumption_forms_for_primes}
\begin{aligned}
& D(g;X;q,C)=\sum_\stacksum{p\leq X}{p \in \Repr(q,C)} w(C,p) - \frac{1}{h(q)}\sum_{K \in \Kcl(q)}\sum_\stacksum{p\leq X}{p \in \Repr(q,K)} w(K,p)  \\
=\ & 
 \pi(X;q,C) e(C) - \frac{1}{h(q)}\sum_{p\leq X} (1+ \chi_q(p)) + O(\log|q|)
\end{aligned}
\end{equation*}
by \eqref{eq:weight_w-Cn_on_average}. Assumption \eqref{eq:bdh_SW_assumption_forms} now follows from \eqref{eq:SW_for_Dirichlet_Char}, the Prime Number Theorem and\linebreak Theorem~\ref{thm:blomer}. 
The term $R(g,Q,X)$ vanishes. Thus, from \eqref{eq:bdh_general} we get 
\begin{equation}\label{eq:bdh2}
	\sum_{q \in \DF(Q)} \sum_{C \in \Kcl(q)} \bigg(\pi(X;q,C) e(C) - \frac{1}{h(q)}\sum_{p\leq X} (1+ \chi_q(p))\bigg)^2\ll_{A,\eps} Q^{1/2}X^2(\log X)^{-A}
	\end{equation}
	if $Q^{3 +\eps} \leq X(\log X)^{-2A-4}$. 
	Similarly to the argument in Section \ref{sec:prelim_1}, one shows that the contribution from exceptional discriminants to the left side of \eqref{eq:bdh2} is negligible (also compare the corresponding argument in the next section). Thus, we may assume the class number bound~\eqref{eq:classno_bound1}. Dyadic decomposition and the large sieve inequality for Dirichlet characters (Lemma~\ref{lem:lsi_org}) 
	 then yield 
	\begin{equation}\label{eq:bdh_use_lsi}
	\sum_{q \in \DF(Q)} \frac{1}{h(q)}\Big(\sum_{p\leq X} \chi_q(p)\Big)^2 \leq (\log Q)^2(Q^{3/2}X+X^2). 
	\end{equation}
	 Therefore, Theorem \ref{thm:bdh1} follows from \eqref{eq:bdh2}, \eqref{eq:bdh_use_lsi} and the Prime Number Theorem if\linebreak $Q \geq (\log X)^{2A+4}$. If $Q$ is smaller, Theorem \ref{thm:bdh1} follows directly from Theorem \ref{thm:blomer}.

\medskip
\begin{remf}
The term $R(g,Q,X)$ clearly vanishes if 
the  function $g$ is supported on primes only or 
if the set $M(Q)$ contains only prime discriminants, for example. 
Thus, we get a clean well-distribution result in these cases. It would be interesting to find other cases in which $R(g,Q,X)$ is dominated by the first term on the right-hand side of \eqref{eq:bdh_general}.
\end{remf}


\section{Proof of the general Barban--Davenport--Halberstam type result}\label{sec:bdh_proof}

We prove Theorem \ref{thm:bdh_general} in this section. 
The proof will be similar to the proofs of the theorems of Section~\ref{sec:bv}. 
First, we consider the contribution coming from the initial range of negative fundamental discriminants. Fix $A>0$. Set $Q_0=(\log X)^{L_0}$ for some $L_0>0$, which will be chosen later and which will depend on $A$ only. By assumption \eqref{eq:bdh_SW_assumption_forms} and the class number bound $h(q) \ll |q|^{1/2}(\log |q|)$, the contribution to the left-hand side of \eqref{eq:bdh_general} coming from discriminants $q$ with $|q|\leq Q_0$ is 
\[
\ll_{L_1} (\log Q_0) Q_0^{3/2} X (\log X)^{-L_1} \sum_{n \leq X} |g(n)|^2 \ll_{L_1} Q_0^{1/2} X(\log X)^{L_0-L_1+1} \sum_{n \leq X} |g(n)|^2
\]
for each $L_1 > L_0$. This is dominated by the right-hand side of \eqref{eq:bdh_general} if 
\begin{equation}\label{eq:bdh_cond_L1}
L_0-L_1+1 \leq -A.
\end{equation}
It remains to consider the large discriminants, i.e.\ all $q$ in
\[
M'(Q):=\{q \in M(Q):\, Q_0 < |q| \leq Q\}
\]
and we may assume from now on that $Q \geq Q_0$.

\medskip
By the definition \eqref{eq:weight_w-Cn_for_primepowers} of the weights $w(C,n)$, we have 
\[
D(g;X;q,C)=\sum_\stacksum{\aF \in B_q(C) \cap Z(q)}{\No(\aF) \leq X} g(\No(\aF)) - \frac{1}{h(q)}\sum_\stacksum{\aF \in Z(q)}{\No(\aF) \leq X} g(\No(\aF)).
\]
For every $q \in \DF$ and every $\chi \in \Hhat(q)$, we set 
\[
G(X;\chi,q):=\sum_\stacksum{\aF \in Z(q)}{\No(\aF) \leq X} g(\No(\aF))\chi(\aF)=\sum_{n \leq X} g(n) \lambda_\chi(n).
\]
By the orthogonality property of ideal class group characters,  
we may rewrite $D(g;X;q,C)$ as
\begin{equation*}
D(g;X;q,C)= 
\frac{1}{h(q)}\sum_{\chi \in \Hhat(q) \smallsetminus\{\chiq_0\}} \overline{\chi}(B_q(C))G(X;\chi,q).
\end{equation*}
Moreover, orthogonality also yields
\[
\sum_{C \in \Hcl(q)} \bigg|\sum_{\chi \in \Hhat(q)\smallsetminus\{\chiq_0\}} \overline{\chi}(C)G(X;\chi,q)\bigg|^2
=h(q) \sum_{\chi \in \Hhat(q)\smallsetminus\{\chiq_0\}} |G(X;\chi,q)|^2.
\]
Thus, 
the contribution from large discriminants to the left-hand side of \eqref{eq:bdh_general} is
\begin{equation*}\label{eq:bdh_erstessplitten}
\sum_{q \in M'(Q)} \sum_{C \in \Hcl(q)} \left|D(g;X; q,B_q^{-1}(C))\right|^2 
= \sum_{q \in M'(Q)} \frac{1}{h(q)} \sum_{\chi \in \Hhat(q)\smallsetminus\{\chiq_0\}} |G(X;\chi,q)|^2.
\end{equation*}
The contribution coming from exceptional discriminants is again negligible if $Q$ is not very small. Indeed, by the bound $|\DF_{\text{ex}}(Q)| \ll \log Q$ (see the proof of Proposition \ref{prop:landau-siegel1}) for the set of exceptional fundamental discriminants $q \in \DF(Q)$, the Cauchy--Schwarz inequality and the bound \eqref{eq:lambda_chi_abschaetzen}, we have
  \[
  \sum_{q\in \DF_{\text{ex}}(Q)}\frac{1}{h(q)} \sum_{\chi \in \Hhat(q)\smallsetminus\{\chiq_0\}} |G(X;\chi,q)|^2 \ll X (\log X)^4 \sum_{n \leq X} |g(n)|^2.
  \]
	In particular, the contribution to the left-hand side of \eqref{eq:bdh_general} coming from exceptional discriminants is negligible if $Q \geq (\log X)^{2A+8}$. This means that we must choose at least
\begin{equation}\label{eq:bdh_cond_L0}
L_0 \geq 2A+8
\end{equation}
 above.

Therefore it remains to estimate the contribution from non-exceptional discriminants, i.e.\ we have to bound
\begin{equation}\label{eq:bdh_gesplittet}
\sum_{q \in M''(Q)} \frac{1}{h(q)} \sum_\stacksum{\chi \in \Hhat(q)}{\chi^2 \neq \chiq_0} |G(X;\chi,q)|^2 \ +\  \sum_{q \in M''(Q)} \frac{1}{h(q)} \sum_\stacksum{\chi \in \Hhat(q)\smallsetminus\{\chiq_0\}}{\chi^2 =\chiq_0} |G(X;\chi,q)|^2, 
\end{equation}
where $M''(Q)=M'(Q) \smallsetminus \DF_{\text{ex}}(Q)$. 

\smallskip

The lower class number bound \eqref{eq:classno_bound1}, dyadic decomposition and the large sieve inequality for complex class group characters (Lemma \ref{lem:lsi_complex})  together imply that the first sum in \eqref{eq:bdh_gesplittet} is bounded above by
\begin{equation}\label{eq:bdh_complex}
\begin{aligned}
& (\log Q) \max_{Q_0 \leq Q_1 \leq Q} Q_1^{-1/2} \sum_{q \in M''(Q_1)} \sum_\stacksum{\chi \in \Hhat(q)}{\chi^2 \neq \chiq_0} |G(X;\chi,q)|^2 \\
\ll_\eps &\ (\log Q) \max_{Q_0 \leq Q_1 \leq Q} Q_1^{-1/2}  \left(X (\log X)^3+X^{1/2}(\log X)Q_1^{5/2+\eps}\right) \sum_{n \leq X} |g(n)|^2
\\
\ll\ &\ Q^{1/2}X^{1/2} \left(X^{1/2}(\log X)^4 Q^{-1/2}Q_0^{-1/2}  + (\log X)^2Q^{3/2+\eps} \right) \sum_{n \leq X} |g(n)|^2
\end{aligned}
\end{equation}
for every $\eps>0$. This is dominated by the right-hand side of \eqref{eq:bdh_general} if $Q \geq (\log X)^{2A+8-L_0}$, which is certainly satisfied if the above-mentioned condition $L_0 \geq 2A+8$ holds.

\smallskip

Like in Section \ref{sec:realchar1}, the second sum in \eqref{eq:bdh_gesplittet} is handled by reducing it to a sum over real Dirichlet characters. If $q \in \DF$ and $\chi  \in \Hhat(q)$ is a real non-trivial class group character, then the Kronecker Factorization Formula implies that $\lambda_\chi(n)$  is the Dirichlet convolution
\begin{equation}\label{eq:realclasschar_convolution}
\lambda_\chi(n)=\chi_{d_1} \ast \chi_{d_2} (n)
\end{equation}
 of two primitive real Dirichlet characters modulo the absolute values of non-trivial fundamental discriminants $d_1$ and $d_2$ with $d_1d_2=q$. Thus, if $\chi \in \Hhat(q)$ is non-trivial and real, then 
\[
G(X;\chi,q)=\sum_{n \leq X} g(n) \sum_{km=n}\chi_{d_1}(k)\chi_{d_2}(m)
\]
for some fundamental discriminants $d_1$ and $d_2$ with $d_1 d_2=q$ and $|d_1|,|d_2|>1$. Moreover, each such pair of discriminants induces one of the non-trivial real class group characters in~$\Hhat(q)$. 

 Let $F(Q)$ denote the set of all fundamental discriminants $d$ with  $1<|d|\leq Q$ (as in Section~\ref{sec:realchar1}). The second sum in~\eqref{eq:bdh_gesplittet} can thus be bounded as follows:
\begin{equation*}\label{eq:bdh_2ndsum_1}
\begin{aligned}
&\sum_{q \in M''(Q)} \frac{1}{h(q)} \sum_\stacksum{\chi \in \Hhat(q)\smallsetminus\{\chi_0\}}{\chi^2 =\chi_0} |G(X;\chi,q)|^2 
\\
=&\  
\sum_{d_1 \in F(Q)} \sum_\stacksum{d_2 \in F(Q)}{d_1d_2 \in M''(Q)} \frac{1}{h(d_1d_2)}\Big|\sum_{n \leq X} g(n) \sum_\stacksum{1 \leq k,m \leq n}{km=n} \chi_{d_1}(k) \chi_{d_2}(m)\Big|^2 
\\
\ll&\ 
(\log Q) \sum_{d_1 \in F(Q)} \sum_\stacksum{d_2 \in F(Q)}{d_1d_2 \in M''(Q)} \frac{1}{|d_1 d_2|^{1/2}} \Big|\sum_{n \leq X} g(n) (\chi_{d_1}(n) + \chi_{d_2}(n))\Big|^2  
\ +\ R(g,Q,X)
\\
\ll&\ 
(\log Q)    \sum_{d_1 \in F(Q)} \frac{1}{|d_1|^{1/2}} |\sum_{n \leq X} g(n) \chi_{d_1}(n)|^2 \sum_{d_2 \leq \frac{Q}{|d_1|}} \frac{1}{d_2^{1/2}} \ +\ R(g,Q,X)
\\
\ll&\ S_1(Q,X) +S_2(Q,X)+ R(g,Q,X),
\end{aligned}
\end{equation*}
where
\[
S_1(Q,X)=Q_0^{1/2}(\log Q) \sum_{d \in F(Q_0)} |\sum_{n \leq X} g(n) \chi_{d}(n)|^2
\]
and
\[
S_2(Q,X)=Q^{1/2}(\log Q)^2 \max_{Q_0 \leq Q_1\leq Q}\ Q_1^{-1} \sum_{d \in F(Q)} |\sum_{n \leq X} g(n) \chi_{d}(n)|^2
\]
 and $R(g,Q,X)$ was defined in Theorem \ref{thm:bdh_general}. 
By positivity and orthogonality, we have 
\begin{equation*}
\begin{aligned}
&S_1(Q,X) \leq Q_0^{1/2}(\log X) \sum_{1<d \leq Q_0} \sum_\stacksum{\chi \mods d}{\chi \neq \chi_0} |\sum_{n \leq X} g(n) \chi(n)|^2
\\
=&\  Q_0^{1/2}(\log X) \sum_{1<d \leq Q_0} \varphi(d) \sum_\stacksum{a \mods d}{(a,d)=1} \Big|\sum_\stacksum{n\leq X}{n \equiv a \mods d} g(n) - \frac{1}{\varphi(d)}\sum_\stacksum{n\leq X}{(n,d)=1} g(n)  \Big|^2.
\end{aligned}
\end{equation*}
By assumption \eqref{eq:bdh_SW_assumption_residues}, we thus have
\begin{equation*}
S_1(Q,X) \ll_{L_2} Q_0^{1/2} X (\log X)^{-L_2+1+3L_0}  \sum_{n \leq X} |g(n)|^2 
\end{equation*}
for all $L_2>L_0$. Hence, $S_1(Q,X)$ is dominated by the right side of \eqref{eq:bdh_general} if 
\begin{equation}\label{eq:bdh_cond_L2}
-L_2+1+3L_0 \leq -A.
\end{equation}

Finally, we use the large sieve inequality for Dirichlet characters, Lemma \ref{lem:lsi_org}, to bound $S_2(Q,X)$. We get
\begin{equation*}
S_2(Q,X) \ll_{L_0} Q^{1/2}(\log X)^2 (Q+XQ_0^{-1})\sum_{n \leq X} |g(n)|^2.
\end{equation*}
This is dominated by the right side of \eqref{eq:bdh_general} if $ Q+XQ_0^{-1} \leq Q^{3/2+\eps}+X(\log X)^{-A-2}$, which is certainly true if the above-mentioned condition $L_0 \geq 2A+8$ holds.

\medskip
By \eqref{eq:bdh_cond_L0}, \eqref{eq:bdh_cond_L1} and \eqref{eq:bdh_cond_L2} we also see that all implied constants above that depend on $L_0$, $L_1$ or $L_2$, can be made dependent on $A$ only, if we choose $L_0=2A+8$, $L_1=A+L_0+1$ and $L_2=A+3L_0+1$, for example. This concludes the proof of Theorem \ref{thm:bdh_general}.

\medskip
\begin{remarkf}\label{rem:BDH-Linde}
If we assume the Lindel\"of Hypothesis, we may use the conditional large sieve inequality of Remark \ref{lem:lsi_complex_conditionalLinde} instead of Lemma \ref{lem:lsi_complex}. Thus, we may then replace the term $Q^{5/2+\eps}$ in the second line of \eqref{eq:bdh_complex}  by $Q^{3/2+\eps}$;   the term $Q^{3/2+\eps}$ in the last line of \eqref{eq:bdh_complex} and in \eqref{eq:bdh_general} may therefore be replaced by $Q^{1/2+\eps}$. Thus, \eqref{eq:bdh1} holds if $Q^{1+\eps} \leq X (\log X)^{-2A-4}$, which yields the second statement in Remark \ref{rem:conditional_ranges}. 
\end{remarkf}


\section{The least prime of the shape \texorpdfstring{$x^2+ny^2$}{x2+ny2}}\label{sec:leastprime}

The statements in Corollary \ref{cor:leastprime} can be proved along the same lines as the analogous results for primes in arithmetic progressions that follow from the Bombieri--Vinogradov theorem and the Barban--Davenport--Halberstam theorem; see \cite{EH71}, for example.

\begin{remarkf}
1.  
From the  
 Siegel--Walfisz theorem for binary quadratic forms, Theorem \ref{thm:blomer}, 
it follows easily that there exists an absolute constant $L$ such that 
\[
\max_{C \in \Kcl(q)} p(q;C) \ll |q|^{L (\log |q|)}
\]
for all $q \in \DF$. 

\noindent  2.  
Kowalski and Michel have proved in \cite{KM02} a log-free zero-density estimate for automorphic forms on $\GL(n)/\Qq$ and described how this can be used to show the existence of an absolute constant $L$ such that
\begin{equation}\label{eq:linnik3}
\max_{C \in \Kcl(q)} p(q;C) \ll |q|^L
\end{equation}
for all $q \in \DF$. This bound is also a consequence of earlier results by Fogels 
\cite{Fog65,Fog67} 
and Weiss~\cite{Wei83}.  
However, no explicit admissible value for $L$ has 
yet been published. 

\noindent  3. 
The Generalized Riemann Hypothesis for 
ideal class group $L$-functions 
implies 
 that \eqref{eq:linnik3} holds for all $q \in \DF$ with $L=1+ \eps$ for all $\eps >0$.

\noindent  4.  
Assuming the Lindel\"of Hypothesis (see Remark \ref{rem:conditional_ranges}), one may replace the exponent $\frac{20}{3}+\eps$ by $2+\eps$ in \eqref{eq:linnik1} and the exponent $3+\eps$ by $1+\eps$ in \eqref{eq:linnik2}.
\end{remarkf}

\smallskip
Focussing on the primes of the special shape $x^2+ny^2$, that is, on primes represented by the principal class of discriminant $-4n$, it is interesting to investigate bounds for the values of $x_{\min{}}$ and $y_{\min{}}$ that yield the smallest prime of this form for any given positive integer~$n$. One would naturally assume that $y_{\min{}}$ is typically very small. Notwithstanding, it is somewhat surprising that numerical calculations even suggest that $y_{\min{}}>1$ can only occur for an exceedingly small set of values~$n$: Up to at least $n=10^8$,
 the smallest prime of the shape $x^2+ny^2$ is actually of the shape $x^2+n$ in all but the eleven cases
\[
n \in \{5,41,59,314,341,479,626,749,755,881,1784\}; 
\]
in all these exceptional cases we have $y_{\min{}}=2$. 
If we could show that $y_{\min{}} =1$ for all $n>1784$ (which appears to be formidable) or could at least get a nice bound for the number/density of exceptions, the problem of bounding the least prime of the shape $x^2+ny^2$ would reduce to bounding the smallest prime of the shape $x^2+n$.  Although this polynomial looks simpler than our original one, there are questions on the prime numbers which it represents that are so much tougher than for binary quadratic forms: There is no integer $n$ for which 
it is nowadays known whether 
there are infinitely many primes of the shape $x^2+n$. Nevertheless, Baier and Zhao \cite{BZ07} proved that, given~$A, B >0$, if $X^2(\log X)^{-A} \leq N \leq X^2$ then
\begin{equation*}\label{eq:baierzhao07}
\sum_\stacksum{n \leq N}{\mu(n)^2 =1} \Big| \sum_{x \leq X} \Lambda(x^2+ n) - \mathfrak{G}(n)X \Big|^2 \ll_{A,B} \frac{NX^2}{(\log X)^B}, 
\ \text{where }\  
\mathfrak{G}(n)=\prod_{p>2} \Bigg(1- \frac{\big(\frac{-n}{p}\big)}{p-1}\Bigg)
\end{equation*}
and $\big(\frac{-n}{p}\big)$ is the Jacobi symbol. Note that $\mathfrak{G}(n)$ converges and $\mathfrak{G}(n) \gg (\log n)^{-1} \gg (\log X)^{-1}$. As in Corollary \ref{cor:leastprime}, we can therefore conclude, from their result and an assumption that appears plausible by our own observations, the following average upper bound for the least prime of the shape $x^2+ny^2$: 

\begin{corollary}\label{cor:leastprime_BZ1} 
Conditional on the 
assumption that $p_0(n)$, the least prime of the shape $x^2+ny^2$, is attained for $y=1$ for all positive squarefree integers $n$ in a set of asymptotic density~$1$, we have
\[
p_0(n)   
\leq n^{2+\eps}  
\]
for all positive squarefree integers $n$ in a set of asymptotic density $1$.
\end{corollary}

\end{document}